\documentclass[final,nomarks]{dmtcs-episciences}
\usepackage{amsmath,amsthm,amssymb}
\usepackage{graphicx}
\usepackage[utf8]{inputenc}
\usepackage{bm}             % boldface symbols (\bm)
\usepackage{graphicx}       % embedding of pictures
\usepackage[usenames]{xcolor}  % typesetting in color
\usepackage{tikz}
\usepackage{pgf}
\usepackage{subfig}

\usepackage{hyperref}

\hypersetup{colorlinks=true,linkcolor=blue, urlcolor=blue, citecolor=blue, 
pdfpagemode=UseNone}

\usetikzlibrary{arrows}

\theoremstyle{plain}
\newtheorem{theorem}{Theorem}[section]
\newtheorem{corollary}[theorem]{Corollary}
\newtheorem{lemma}[theorem]{Lemma}

\newtheorem{conjecture}[theorem]{Conjecture}

\newtheorem{observation}[theorem]{Observation}
\theoremstyle{definition}
\newtheorem{definition}[theorem]{Definition}

\newtheorem{claim}{Claim}
\theoremstyle{remark}

\newcommand{\ds}{\mathrm{D}} %dented shape
\newcommand{\fd}{\phi_{\ds}} %filling of the dented shape

\newcommand{\cG}{\mathcal{G}}
\newcommand{\cgi}{\cG_i}
\newcommand{\cgip}{\cG_{i+1}}

\newcommand{\bbZ}{\mathbb{Z}}
\newcommand{\F}{\mathcal{F}}

\DeclareMathOperator{\Tr}{Tr}

\title{Fillings of skew shapes avoiding diagonal patterns\thanks{Supported by project 18-19158S of 
the Czech Science Foundation and by GAUK project 1766318}}

\author{V{\'\i}t Jel{\'\i}nek \and Mark Karpilovskij}
\affiliation{Computer Science Institute, Charles University, Prague}


\received{2020-02-28}

\revised{2021-05-27}

\accepted{2021-05-28}

\publicationdetails{22}{2021}{2}{8}{6171}

\begin{document}
\maketitle
\begin{abstract}
A skew shape is the difference of two top-left justified Ferrers shapes sharing the same top-left 
corner. We study integer fillings of skew shapes. As our first main result, we show that 
for a specific hereditary class of skew shapes, which we call $\ds$-free shapes, the fillings that 
avoid a north-east chain of size $k$ are in bijection with fillings that avoid a south-east chain of 
the same size. Since Ferrers shapes are a subclass of $\ds$-free shapes, this result can be seen as 
a generalization of previous analogous results for Ferrers shapes. 

As our second main result, we construct a bijection between 01-fillings of an arbitrary skew shape 
that avoid a south-east chain of size 2, and the 01-fillings of the same shape that simultaneously 
avoid a north-east chain of size 2 and a particular non-square subfilling. This generalizes a 
previous result for transversal fillings.
\end{abstract}

\section{Introduction}
Wilf equivalence, and the closely related notion of Wilf order, are undoubtedly among the key 
notions in the study of combinatorics of permutations. Despite considerable amount of research, 
many key questions related to these concepts remain open. 

One of the most successful approaches in the study of Wilf equivalence is based on the concept of 
pattern avoidance in fillings of specific shapes (see Section~\ref{sec-prelim} for precise 
definitions). Of particular interest are the so-called transversal fillings, which contain exactly 
one 1-cell in each row and each column, and can be viewed as generalizations of permutation 
matrices. Transversal fillings of Ferrers shapes have played a crucial part in establishing most of 
the known results on Wilf equivalence~\cite{bwx,SW}, as well as in the study of many related 
combinatorial objects, such as involutions \cite{BMS}, words~\cite{GTZ}, matchings~\cite{Dyck,KZ}, 
graphs~\cite{DM2,DM} or set partitions~\cite{Chen,JM}. 

Most of the research into pattern-avoiding fillings focuses on the avoidance of diagonal 
patterns. We let $\delta_k$ and $\iota_k$ denote the decreasing and increasing diagonal pattern of 
size $k$, respectively. Backelin, West and Xin~\cite{bwx} have shown that the number of transversal 
fillings of a Ferrers shape that avoid $\delta_k$ is the same as the number of such fillings that 
avoid~$\iota_k$. This was later generalized to non-transversal fillings~\cite{Krattenthaler06}, and 
to more general shapes, such as stack polyominoes~\cite{GuPo}, moon polyominoes~\cite{Rubey11}, or 
almost-moon 
polyominoes~\cite{PoYa}.

In this paper, we consider yet another family of shapes, namely the so-called skew shapes, 
which can be seen as a difference of two Ferrers shapes, or equivalently, as a shape whose 
boundary is a union of two nondecreasing lattice paths with common endpoints. Our 
interest in skew shapes is motivated by the following conjecture.

\begin{conjecture}[Skew shape conjecture] For any skew shape $S$ and any number $k$, the number of 
transversal fillings of $S$ that avoid $\iota_k$ is greater than or equal to the number of such 
fillings that avoid~$\delta_k$.
\end{conjecture}

Verifying the skew shape conjecture for any particular value of $k$ would imply that for 
any two permutations $\alpha$ and $\beta$, the pattern $\alpha\ominus\delta_k\ominus\beta$ is more 
Wilf-restrictive than $\alpha\ominus\iota_k\ominus\beta$, where $\ominus$ denotes the skew-sum of 
permutations (see, e.g., Vatter's survey~\cite{Vatter} for permutation-related terminology). 
However, the conjecture has so far only been verified for $k\le 2$~\cite{bp,phd}. 

In this paper, we prove two new results on diagonal-avoiding fillings of skew shapes, both of which 
indirectly support the skew shape conjecture. As our first main result, we will show that the 
conjecture holds for skew shapes that avoid the shape $\ds$ from Figure~\ref{fig-D} as subshape. In 
fact, we will show that for $\ds$-avoiding skew shapes the conjecture holds with equality, via a 
bijection that extends to general integer fillings. Since a Ferrers shape is a special case of a 
$\ds$-avoiding skew shape, this result can be seen as another generalization of the result of 
Backelin et al.~\cite{bwx} on transversals of Ferrers shapes and its generalization by 
Krattenthaler~\cite{Krattenthaler06} to non-transversal fillings. 

\begin{figure}
\centerline{\includegraphics{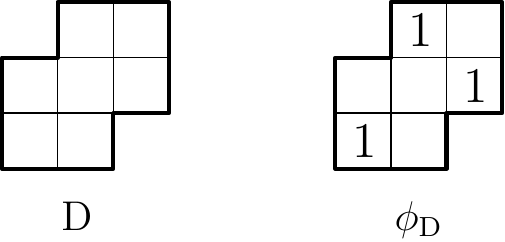}}
 \caption{The shape $\ds$ (left) and its filling $\fd$ (right).}\label{fig-D}
\end{figure}

As our second main result, we show that for any skew shape $S$, the number of 01-fillings of $S$ 
avoiding the pattern $\delta_2$ is the same as the number of 01-fillings of $S$ that avoid
both $\iota_2$ and the pattern $\fd$ from Figure~\ref{fig-D}. This has been previously known for 
transversal fillings, where the proof follows from the fact that each skew shape that admits at 
least one transversal filling admits exactly one $\delta_2$-avoiding filling and exactly one 
$\{\iota_2,\fd\}$-avoiding filling. To extend the result to general, i.e., non-transversal, 
fillings, 
we devise a new bijective argument.

\section{Preliminaries}\label{sec-prelim}

\paragraph{Shapes.}
A \emph{cell} is a unit square whose vertices are points of the $\bbZ^2$ integer grid, and a 
\emph{shape} is the union of finitely many cells. The \emph{coordinates} of a cell are the 
Cartesian coordinates of its top-right corner. If two shapes differ only by a translation in the 
plane, we shall treat them as identical. We will henceforth assume that in each nonempty shape
we consider, the leftmost cell has horizontal coordinate equal to 1, and the bottommost cell has 
vertical coordinate equal to~1, unless explicitly stated otherwise. We use the term \emph{cell 
$(i,j)$} for the cell in column $i$ and row~$j$. The \emph{height} of a shape is the vertical 
coordinate of its topmost cell, and its \emph{width} is the horizontal coordinate of 
its rightmost cell. We say that a cell $(i,j)$ is \emph{strictly to the left} of a cell $(i',j')$ if 
$i<i'$, and we say that it is \emph{weakly to the left} if $i\le i'$. We will also speak of 
$(i',j')$ being weakly (or strictly) to the right, above, or below $(i,j)$, with obvious meanings. 

Two cells are \emph{adjacent} if their boundaries share an edge. For a shape $S$, the transitive 
closure of the adjacency relation on the cells of $S$ yields an equivalence, whose blocks are the 
\emph{connected components} of~$S$. If $S$ has a single component, we say that it is 
\emph{connected}. Connected shapes are known as \emph{polyominoes}.

\begin{figure}[ht]
\centering
\includegraphics[width=0.8\textwidth]{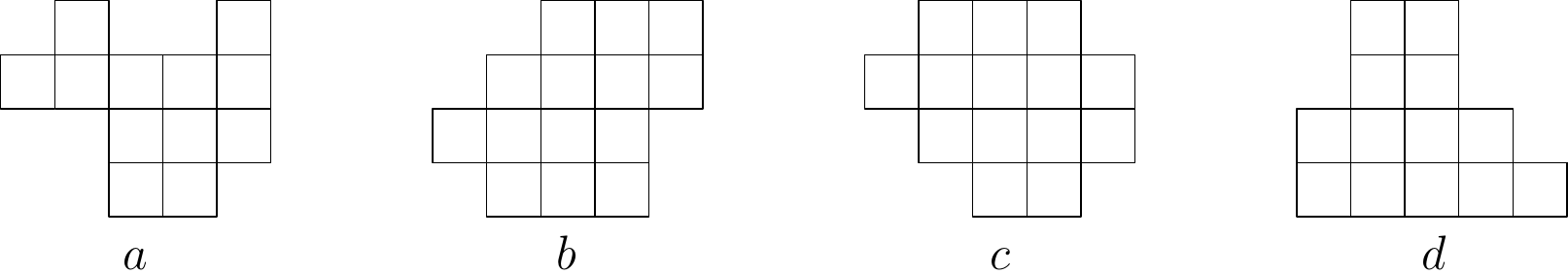}
\caption{Examples of shapes: (a) non-convex, (b) convex but not intersection-free, (c) moon 
polyomino, (d) bottom-justified moon polyomino.}
\label{fig-examples}
\end{figure}

A shape is \emph{convex} if for any two of its cells in the same row or column it also contains 
every cell in between. Two columns of a shape are \emph{comparable} if the set of row coordinates of 
one column is contained in the set of row coordinates of the other; comparable rows are defined 
analogously. We say that a shape is \emph{intersection-free} if any two of its columns are 
comparable. Note that this is equivalent to having any two of its rows comparable. Convex 
intersection-free shapes are known as \emph{moon polyominoes} (see Figure~\ref{fig-examples}). 

A shape is \emph{top-justified} if the topmost cells of all of its columns are in the same row and 
it is \emph{left-justified} if the leftmost cells of all of its rows are in the same column. 
\emph{Bottom-justified} and \emph{right-justified} shapes are defined analogously. A shape 
is \emph{top-left justified} if it is both top-justified and left-justified, and similarly for other 
combinations of horizontal and vertical directions.

A \emph{northwest Ferrers shape}, or \emph{NW Ferrers shape} for short, is a top-left justified 
moon polyomino, and similarly, a \emph{SE Ferrers shape} is a bottom-right justified moon 
polyomino; see Figure~\ref{fig_nw_se}.

\begin{figure}[ht]
\centering
% \subfloat[][NW Ferrers shape] {
% \begin{tikzpicture}[line cap=round,line join=round,>=triangle 45,x=1.0cm,y=1.0cm]
% \clip(1.6,0) rectangle (8.24,6.22);
% \draw (2,1)-- (2,6);
% \draw (2,6)-- (8,6);
% \draw (2,1)-- (4,1);
% \draw (4,1)-- (4,2);
% \draw (4,2)-- (5,2);
% \draw (5,2)-- (5,3);
% \draw (8,6)-- (8,4);
% \draw (8,4)-- (7,4);
% \draw (7,4)-- (7,3);
% \draw (7,3)-- (5,3);
% 
% \draw (3,6)-- (3,1);
% \draw (4,6)-- (4,2);
% \draw (5,6)-- (5,3);
% \draw (6,6)-- (6,3);
% \draw (7,6)-- (7,4);
% \draw (8,5)-- (2,5);
% \draw (7,4)-- (2,4);
% \draw (5,3)-- (2,3);
% \draw (4,2)-- (2,2);
% 
% \end{tikzpicture}
% }
% \subfloat[][SE Ferrers shape] {
% \begin{tikzpicture}[line cap=round,line join=round,>=triangle 45,x=1.0cm,y=1.0cm]
% \clip(10.66,0) rectangle (16.45,6.9);
% \draw (11,1)-- (16,1);
% \draw (16,1)-- (16,6);
% \draw (11,1)-- (11,2);
% \draw (11,2)-- (12,2);
% \draw (12,2)-- (12,5);
% \draw (12,5)-- (14,5);
% \draw (14,5)-- (14,6);
% \draw (14,6)-- (16,6);
% \draw (12,2)-- (16,2);
% \draw (12,3)-- (16,3);
% \draw (12,4)-- (16,4);
% \draw (14,5)-- (16,5);
% \draw (15,6)-- (15,1);
% \draw (12,2)-- (12,1);
% \draw (13,5)-- (13,1);
% \draw (14,5)-- (14,1);
% \end{tikzpicture}
% }
\includegraphics[width=0.8\textwidth]{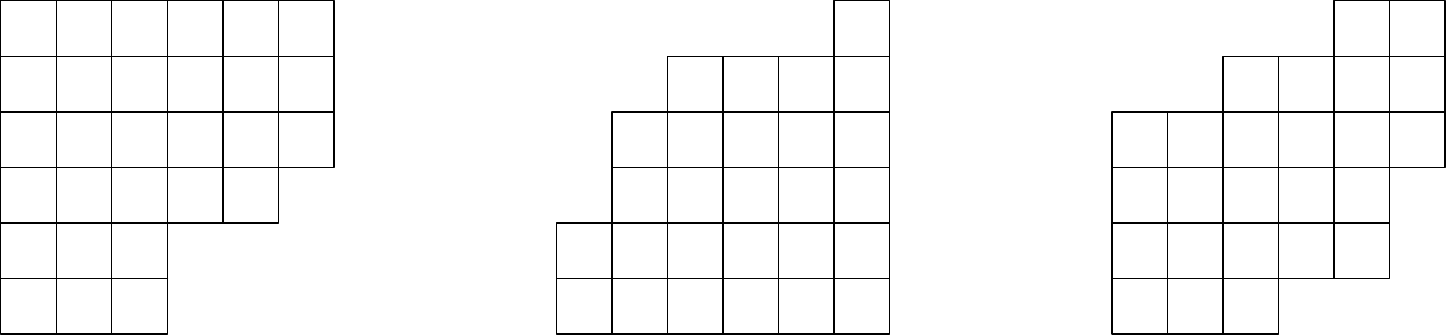}
\caption{Three types of shapes relevant for our paper. Left: a NW Ferrers shape; middle: a SE 
Ferrers shape; right: a skew shape.}
\label{fig_nw_se}
\end{figure}

A \emph{skew shape} is a shape $S$ that can be obtained by taking two NW Ferrers shapes $F_1$ 
and $F_2$ with a common top-left cell and putting $S=F_1\setminus F_2$, i.e., $S$ is the union of 
those cells of $F_1$ that do not belong to~$F_2$. An equivalent characterisation of skew shapes is 
provided by the next observation.

\begin{observation}\label{obs-skew}
 A shape $S$ is a skew shape if and only for any two columns $i_1<i_2$ and any two rows $j_1<j_2$, 
if $S$ contains both $(i_1,j_2)$ and $(i_2,j_1)$, then $S$ contains all the cells from the 
set $\{(i,j);\; i_1\le i\le i_2 \text{ and } j_1\le j\le j_2\}$.
\end{observation}

\paragraph{Fillings of shapes.}
A \emph{filling} of a shape $S$ is a function $\phi$ that assigns a nonnegative integer to every 
cell of~$S$. We write $\phi(i,j)$ for the value assigned by $\phi$ to a cell~$(i,j)$.  A 
\emph{binary filling} (also known as \emph{01-filling}) is a filling that only uses the values 0 
and~1. A cell that has been assigned the value 0 (or 1) by a binary filling $\phi$ is referred to as 
a \emph{0-cell} (or \emph{1-cell}, respectively) of~$\phi$. If a set $X$ of cells of a shape $S$ 
contains at least one 1-cell of a filling $\phi$, we say that $X$ is \emph{nonzero} in $\phi$, 
written as $X\neq 0$; otherwise $X$ is \emph{zero}, and we write $X=0$.

A binary filling $\phi$ of a shape $S$ is \emph{sparse} if every row and every column of $S$ has at 
most one 1-cell. A binary filling $\phi$ is a \emph{transversal filling} (or just a 
\emph{transversal}) if every row and every column of $S$ has exactly one 1-cell. 

In figures, we will often represent 0-cells by empty boxes and 1-cells by boxes with crosses; see 
Figure~\ref{figure_fillings} for examples.

\begin{figure}
\centering
\subfloat[][A transversal filling] {
\begin{tikzpicture}[line cap=round,line join=round,>=triangle 45,x=0.75cm,y=0.75cm]
\clip(-0.21,-1.19) rectangle (6.12,5.2);
\draw (1,4)-- (1,0);
\draw (1,0)-- (5,0);
\draw (5,0)-- (5,4);
\draw (5,4)-- (1,4);
\draw (2,0)-- (2,5);
\draw (3,0)-- (3,5);
\draw (4,0)-- (4,5);
\draw (2,5)-- (5,5);
\draw (5,5)-- (5,4);
\draw (3,0)-- (3,-1);
\draw (3,-1)-- (6,-1);
\draw (6,-1)-- (6,2);
\draw (6,2)-- (5,2);
\draw (5,1)-- (0,1);
\draw (5,2)-- (0,2);
\draw (5,3)-- (0,3);
\draw (0,1)-- (0,4);
\draw (0,4)-- (1,4);
\draw (4,0)-- (4,-1);
\draw (5,0)-- (5,-1);
\draw (5,0)-- (6,0);
\draw (5,1)-- (6,1);
\draw (0.5,3.5) node {$\times$};
\draw (2.5,2.5) node {$\times$};
\draw (4.5,1.5) node {$\times$};
\draw (1.5,0.5) node {$\times$};
\draw (3.5,4.5) node {$\times$};
\draw (5.5,-0.5) node {$\times$};
\end{tikzpicture}
}
\subfloat[][A non-sparse filling] {
\begin{tikzpicture}[line cap=round,line join=round,>=triangle 45,x=0.75cm,y=0.75cm]
\clip(-0.21,-1.19) rectangle (6.12,5.2);
\draw (1,4)-- (1,0);
\draw (1,0)-- (5,0);
\draw (5,0)-- (5,4);
\draw (5,4)-- (1,4);
\draw (2,0)-- (2,5);
\draw (3,0)-- (3,5);
\draw (4,0)-- (4,5);
\draw (2,5)-- (5,5);
\draw (5,5)-- (5,4);
\draw (3,0)-- (3,-1);
\draw (3,-1)-- (6,-1);
\draw (6,-1)-- (6,2);
\draw (6,2)-- (5,2);
\draw (5,1)-- (0,1);
\draw (5,2)-- (0,2);
\draw (5,3)-- (0,3);
\draw (0,1)-- (0,4);
\draw (0,4)-- (1,4);
\draw (4,0)-- (4,-1);
\draw (5,0)-- (5,-1);
\draw (5,0)-- (6,0);
\draw (5,1)-- (6,1);
\draw (1.5,0.5) node {$\times$};
\draw (1.5,2.5) node {$\times$};
\draw (2.5,2.5) node {$\times$};
\draw (3.5,2.5) node {$\times$};
\draw (3.5,3.5) node {$\times$};
\draw (4.5,4.5) node {$\times$};
\end{tikzpicture}
}
\caption{Examples of binary fillings}
\label{figure_fillings}
\end{figure}
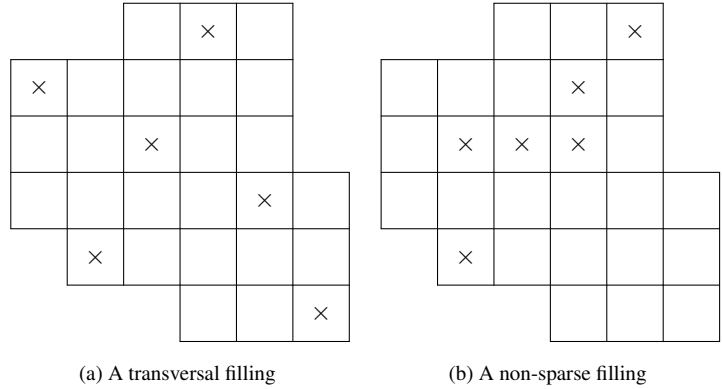

\begin{figure}[tb]
 \centerline{\includegraphics[width=0.7\textwidth]{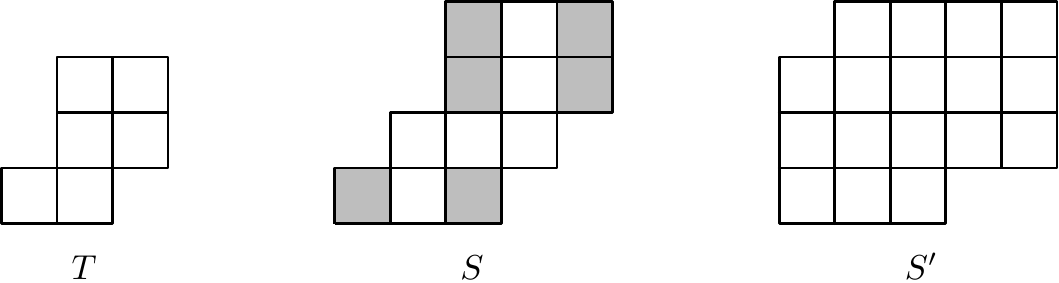}}
\caption{Illustration of shape containment. Left: a shape $T$. Middle: a shape $S$ which contains 
$T$ as a subshape. The shaded cells correspond to an occurrence of $T$ formed by columns $(1,3,5)$ 
and rows $(1,3,4)$. Right: a shape $S'$ that does not contain~$T$.}\label{fig-contain}
\end{figure}

\paragraph{Containment of shapes and fillings.}
Let $S$ be a shape of height $h$ and width $w$, and let $T$ be a shape of height $k$ and 
width~$\ell$. We say that $S$ \emph{contains} $T$, or that $T$ is a \emph{subshape} of $S$, if there 
is a sequence of row indices $1\le r_1 < r_2 <\dotsb < r_k\le h$ and column indices $1\le c_1 < c_2 
<\dotsb<c_\ell\le w$, such that for every $i$ and $j$, the cell $(i,j)$ is contained in $T$ if and 
only if the cell $(c_i,r_j)$ is in~$S$. The rows $(r_1,\dotsc,r_k)$ and columns 
$(c_1,\dotsc,c_\ell)$ together form an \emph{occurrence} of $T$ in~$S$. If $S$ does not contain 
$T$, we say that $S$ \emph{avoids}~$T$. See Figure~\ref{fig-contain} for an example.

With $S$ and $T$ as above, if $\phi_S$ is a filling of $S$ and $\phi_T$ a filling of 
$T$, we say that $\phi_S$ \emph{contains} $\phi_T$, or that $\phi_T$ is a \emph{subfilling} of 
$\phi_S$, if there is an occurrence of $T$ in $S$ in rows $(r_1,\dotsc,r_k)$ and columns 
$(c_1,\dotsc,c_\ell)$ with the additional property that for every cell $(i,j)$ of $T$, we have 
$\phi_T(i,j)\le\phi_S(c_i,r_j)$. We again say that the rows $(r_1,\dotsc,r_k)$ and columns 
$(c_1,\dotsc,c_\ell)$ form an \emph{occurrence} of $\phi_T$ in~$\phi_S$. In the case when $\phi_T$ 
is binary, the condition $\phi_T(i,j)\le\phi_S(c_i,r_j)$ simply says that each 1-cell of $\phi_T$ 
must get mapped to a nonzero cell of $\phi_S$ by the occurrence.

Note that any subshape of a moon polyomino is again a moon polyomino, any subshape of a skew shape 
is a skew shape, and any subshape of a Ferrers shape is a Ferrers shape. 

Let $S_n$ be the square shape of size $n\times n$, i.e., the shape consisting of the cells $(i,j)$ 
with $1\le i\le n$ and $1\le j\le n$. The transversal fillings of $S_n$ correspond in a natural way 
to permutations of order $n$: given a permutation $\pi=\pi_1\pi_2\dotsb\pi_n$ of order $n$, 
represented by a sequence of numbers in which every value from 1 to $n$ appears exactly once, we 
represent $\pi$ by a transversal filling of $S_n$ whose 1-cells are the cells $(i,\pi_i)$ for 
$i=1,\dotsc,n$. Such a transversal filling is usually called the \emph{permutation diagram} 
of~$\pi$. In this paper, we shall make no explicit distinction between a permutation and its diagram, 
and if there is no risk of ambiguity, we will use the term permutation to refer both to the sequence 
of integers and to the diagram. Note that the containment relation defined above for binary fillings 
is a generalization of the classical Wilf containment of permutations. 

In this paper, we shall be mostly working with two types of forbidden patterns, corresponding to 
increasing and decreasing permutations, respectively. For any $k\ge 1$, let $\iota_k$ denote the 
diagram of the increasing permutation $1,2,3,\dotsc, k$; in other words, $\iota_k$ is the filling 
of the $k\times k$ square shape with 1-cells $(i,i)$ for $1\le i\le k$. Symmetrically, let 
$\delta_k$ be the diagram of the decreasing permutation $k,k-1,\dotsc,1$. An occurrence of $\iota_k$ 
in a filling $\phi$ is also referred to as a \emph{NE-chain of length $k$} in $\phi$, while an 
occurrence of $\delta_k$ is a \emph{SE-chain of length $k$}. See Figure \ref{figure_chains} for 
examples.

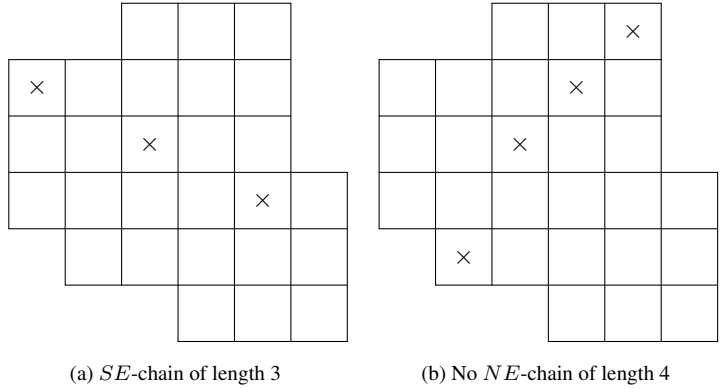
\begin{figure}
\centering
\subfloat[][$SE$-chain of length 3] {
\begin{tikzpicture}[line cap=round,line join=round,>=triangle 45,x=0.75cm,y=0.75cm]
\clip(-0.21,-1.19) rectangle (6.12,5.2);
\draw (1,4)-- (1,0);
\draw (1,0)-- (5,0);
\draw (5,0)-- (5,4);
\draw (5,4)-- (1,4);
\draw (2,0)-- (2,5);
\draw (3,0)-- (3,5);
\draw (4,0)-- (4,5);
\draw (2,5)-- (5,5);
\draw (5,5)-- (5,4);
\draw (3,0)-- (3,-1);
\draw (3,-1)-- (6,-1);
\draw (6,-1)-- (6,2);
\draw (6,2)-- (5,2);
\draw (5,1)-- (0,1);
\draw (5,2)-- (0,2);
\draw (5,3)-- (0,3);
\draw (0,1)-- (0,4);
\draw (0,4)-- (1,4);
\draw (4,0)-- (4,-1);
\draw (5,0)-- (5,-1);
\draw (5,0)-- (6,0);
\draw (5,1)-- (6,1);
\draw (0.5,3.5) node {$\times$};
\draw (2.5,2.5) node {$\times$};
\draw (4.5,1.5) node {$\times$};
\end{tikzpicture}
}
\subfloat[][No $NE$-chain of length 4] {
\begin{tikzpicture}[line cap=round,line join=round,>=triangle 45,x=0.75cm,y=0.75cm]
\clip(-0.21,-1.19) rectangle (6.12,5.2);
\draw (1,4)-- (1,0);
\draw (1,0)-- (5,0);
\draw (5,0)-- (5,4);
\draw (5,4)-- (1,4);
\draw (2,0)-- (2,5);
\draw (3,0)-- (3,5);
\draw (4,0)-- (4,5);
\draw (2,5)-- (5,5);
\draw (5,5)-- (5,4);
\draw (3,0)-- (3,-1);
\draw (3,-1)-- (6,-1);
\draw (6,-1)-- (6,2);
\draw (6,2)-- (5,2);
\draw (5,1)-- (0,1);
\draw (5,2)-- (0,2);
\draw (5,3)-- (0,3);
\draw (0,1)-- (0,4);
\draw (0,4)-- (1,4);
\draw (4,0)-- (4,-1);
\draw (5,0)-- (5,-1);
\draw (5,0)-- (6,0);
\draw (5,1)-- (6,1);
\draw (1.5,0.5) node {$\times$};
\draw (2.5,2.5) node {$\times$};
\draw (3.5,3.5) node {$\times$};
\draw (4.5,4.5) node {$\times$};
\end{tikzpicture}
}
\caption{Examples of chains in a shape with a binary filling. Left: an example of a SE-chain 
of length 3. Right: a filling with no NE-chain of length 4, but with two NE-chains of length 3. 
Note that the four 1-cells in the right figure do not form a single NE-chain of length 4, because 
the subshape spanned by the four 1-cells is not a square shape: the column of the leftmost 1-cell 
does not intersect the row of the topmost 1-cell inside the shape.}
\label{figure_chains}
\end{figure}

\paragraph{Results.} Our first main result deals with transversal fillings of skew shapes of a 
special type. Let $\ds$ be the skew shape in Figure~\ref{fig-D} ($\ds$ stands for `dented shape'). We 
say that a skew shape $S$ is \emph{$\ds$-free} if $S$ avoids~$\ds$. 

Let $\Tr(S,\pi)$ be the number of transversal fillings of the shape $S$ that avoid the 
filling~$\pi$. Backelin et al.~\cite{bwx} have shown that for any $k$ and any Ferrers shape $F$, we 
have the identity $\Tr(F,\delta_k)=\Tr(F,\iota_k)$. This was later generalized by 
Krattenthaler~\cite{Krattenthaler06} to general fillings of Ferrers shapes. Our first main result 
generalizes this identity to a broader class of shapes, namely to $\ds$-free skew shapes. 

\begin{theorem}\label{thm-sskew}
For any $\ds$-free skew shape $S$, there is a bijection transforming a filling $\phi$ of $S$ to a 
filling $\phi'$ of $S$ with these properties.
\begin{itemize}
\item For any $k\ge 1$,  $\phi$ avoids $\delta_k$ if and only if $\phi'$ avoids~$\iota_k$.
\item There is a permutation $\rho$ of row-indices of $S$ and a permutation $\sigma$ of 
column-indices of $S$ such that for any $i$, the entries in the $i$-th row of $\phi$ have the same 
sum as the entries in the $\rho_i$-th row of $\phi'$, and the entries in the $i$-th column of $\phi$ 
have the same sum as the entries in the $\sigma_i$-th row of~$\phi'$.
\end{itemize}
\end{theorem}
The proof of this result is presented in Section~\ref{sec-sskew}.  A direct consequence of the 
theorem is the following identity for transversals.

\begin{corollary}\label{cor-sskew}
For any $\ds$-free skew shape $S$ and any $k\ge 1$, we have $\Tr(S,\delta_k)=\Tr(S,\iota_k)$.
\end{corollary}

Theorem~\ref{thm-sskew} cannot in general be extended to non-$\ds$-free skew shapes. In fact, 
for the shape $\ds$ itself, we easily verify that $\Tr(\ds,\iota_2)=2$ and $\Tr(\ds,\delta_2)=1$. 
However, computational evidence suggests the following conjecture. 

\begin{conjecture}\label{con-skew}[Skew shape conjecture]
For any skew shape $S$ and any $k\ge 1$, we have $\Tr(S,\iota_k)\ge\Tr(S,\delta_k)$.
\end{conjecture}

The conjecture is trivial for $k=1$ and is known to be true for $k=2$. Indeed, for the case $k=2$, 
a more precise result is known, proven by Jelínek~\cite[Lemmas 29 and 30]{phd} and independently by 
Burstein and Pantone~\cite[Lemmas 1.4 and 1.5]{bp}.

\begin{theorem}[\cite{bp,phd}]\label{thm-bp}
For any skew shape $S$ that admits at least one transversal, there is exactly one transversal of $S$ 
that avoids $\delta_2$, and exactly one transversal of $S$ that simultaneously avoids both 
$\iota_2$ and the filling $\fd$ from Figure~\ref{fig-D}.
\end{theorem}

Note that this theorem implies not only the case $k=2$ of the skew shape conjecture, but also the 
case $k=2$ of Corollary~\ref{cor-sskew}. 

As our second main result, we will extend Theorem~\ref{thm-bp} to an identity for general (i.e., not 
necessarily transversal) binary fillings. 

\begin{theorem}\label{thm-genskew}
For any skew shape $S$, the number of binary fillings of $S$ that avoid $\delta_2$ is the same as 
the number of binary fillings of $S$ that avoid both $\iota_2$ and~$\fd$.  Moreover, this identity is 
witnessed by a bijection that preserves the number of 1-cells in every row.
\end{theorem}
The proof of this result appears in Section~\ref{sec-genskew}. We note that the original approach 
used to prove Theorem~\ref{thm-bp} does not seem to be applicable to non-transversal fillings. 
Instead, our proof of Theorem~\ref{thm-genskew} is based on a new bijective argument.

\section{Proof of Theorem~\texorpdfstring{\ref{thm-sskew}}{2.2}}\label{sec-sskew}

\subsection{The structure of \texorpdfstring{$\ds$}{dented-shape}-free skew shapes}
An important feature of $\ds$-free skew shapes is that they admit a natural decomposition into 
a concatenation of Ferrers diagrams. Before describing this decomposition properly, we need some 
preparation.

Let $S$ be a shape, and let $(i,j)$ be a cell of $S$. We let $S[\le i, \le j]$ denote the subshape 
of $S$ formed by those cells $(i',j')$ of $S$ that satisfy $i'\le i$ and $j'\le j$. We also use the 
analogous notations $S[\ge i, \le j]$, $S[\le i, \ge j]$, and $S[\ge i, \ge j]$, with obvious 
meanings.

\begin{lemma} \label{lemma_forb_rec}
A skew shape $S$ is $\ds$-free if and only if for every cell $(i,j)$ of $S$ at least one
of the two subshapes $S[\le i, \ge j]$ and $S[\ge i, \le j]$ is a rectangle.
\end{lemma}
\begin{proof}
Suppose that a skew shape $S$ has an occurrence of $\ds$ in columns $(i_1,i_2,i_3)$ and rows 
$(j_1,j_2,j_3)$. Then $S[\le i_2,\ge j_2]$ is not a rectangle, since it contains the two cells 
$(i_1,j_2)$ and $(i_2,j_3)$ but does not contain $(i_1,j_3)$, which is not a cell of~$S$. 
Symmetrically, $S[\ge i_2,\le j_2]$ is not a rectangle either, and the right-hand side of the 
equivalence in the statement of the lemma fails for $i=i_2$ and $j=j_2$.

Conversely, suppose that $S$ contains a cell $(i,j)$ such that neither $S[\le i, \ge j]$ nor $S[\ge 
i, \le j]$ is a rectangle; see Figure~\ref{fig-forbrec}. Fix row indices $j^-$ and $j^+$ so that 
$(i,j^-)$ is the bottommost cell of $S$ in column $i$, and $(i,j^+)$ is the topmost cell of $S$ in 
column~$i$. Note that $j^-<j<j^+$: if we had, e.g., $j^-=j$, then $S[\ge i, \le j]$ would consist of 
a single row, and therefore it would be a rectangle. Fix column indices $i^-$ and $i^+$ so that 
$(i^-,j)$ and $(i^+,j)$ are the leftmost and rightmost cell of $S$ in row $j$. Again, we see that 
$i^-<i<i^+$.

\begin{figure}
 \centerline{\includegraphics[scale=0.7]{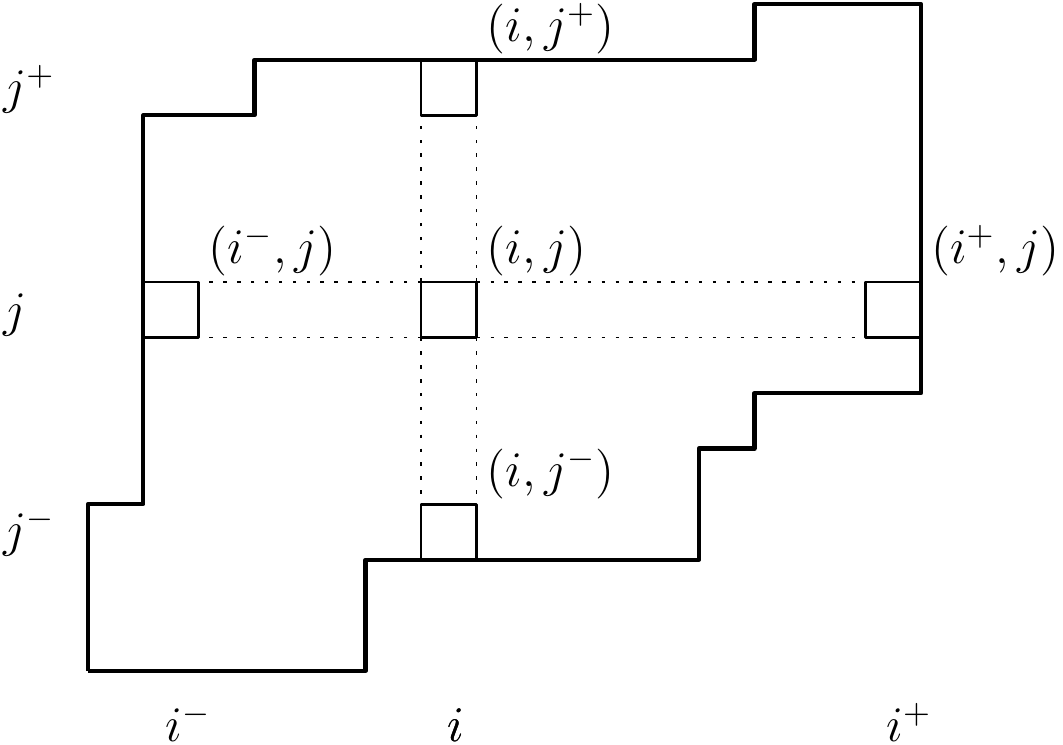}}
\caption{Finding an occurrence of $\ds$ in a skew shape}\label{fig-forbrec}
\end{figure}

Note that neither $(i^-,j^+)$ nor $(i^+,j^-)$ is a cell of $S$: if, e.g., $S$ contained the cell 
$(i^-,j^+)$, then $S[\le i, \ge j]$ would be a rectangle. On the other hand, since $S$ is a skew 
shape, it must contain both $(i^-,j^-)$ and $(i^+,j^+)$ by Observation~\ref{obs-skew}. It follows 
that the columns $(i^-,i,i^+)$ and rows $(j^-,j,j^+)$ induce an occurrence of $\ds$ in~$S$.
\end{proof}

Let $S$ be a connected skew shape, and let a vertical line placed between two adjacent columns of 
$S$ divide it into two skew shapes $S_1$ and $S_2$, where the leftmost column of $S_2$ is directly 
to the right of the rightmost column of~$S_1$, and the bottommost row of $S_2$ is at least as high 
as the bottommost cell of the rightmost column of~$S_1$. Then we say that $S$ is a \emph{vertical 
concatenation} of $S_1$ and $S_2$ and we write $S = S_1 |^v S_2$. Similarly, we define 
\emph{horizontal concatenation} and write $T = T_1 |^h T_2$.

\begin{figure}[tb]
\centering
\begin{tikzpicture}[line cap=round,line join=round,>=triangle 45,x=0.75cm,y=0.75cm]
\clip(-4.25,-4.33) rectangle (6.19,4.2);
\draw (-4,-4)-- (-4,0);
\draw (-4,0)-- (0,0);
\draw (0,0)-- (0,-2);
\draw (0,-2)-- (-1,-2);
\draw (-1,-2)-- (-1,-3);
\draw (-1,-3)-- (-2,-3);
\draw (-2,-3)-- (-2,-4);
\draw (-2,-4)-- (-4,-4);
\draw (0,0)-- (0,1);
\draw (0,1)-- (2,1);
\draw (2,1)-- (2,2);
\draw (0,-1)-- (4,-1);
\draw (2,2)-- (4,2);
\draw (4,2)-- (4,-1);
\draw (-2.8,-1.35) node[anchor=north west] {$F_1$};
\draw (1.64,0.42) node[anchor=north west] {$G_1$};
\draw (2,2)-- (5,2);
\draw (3,2)-- (3,4);
\draw (3,4)-- (6,4);
\draw (6,4)-- (6,3);
\draw (6,3)-- (5,3);
\draw (5,3)-- (5,2);
\draw (3.5,3.35) node[anchor=north west] {$F_2$};
\end{tikzpicture}
\caption{A $\ds$-free skew skew shape, with its Ferrers decomposition}\label{fig-decomp1}
\end{figure}
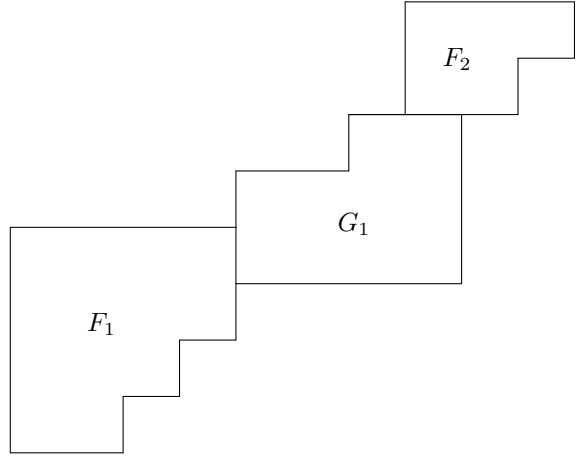

We say that a skew shape $S$ is \emph{Ferrers-decomposable} if it can be written as 
\[S = F_1 |^v G_1 |^h F_2 |^v \cdots |^h F_n |^v G_n,\]
with the following properties (see Figure~\ref{fig-decomp1}):
\begin{itemize}
\item[(a)] All $F_i$ are $NW$ Ferrers shapes, and all $G_i$ are $SE$ Ferrers shapes.
\item[(b)] All $F_i$ and $G_i$ are nonempty with the possible exception of $G_n$.
\item[(c)] For $i<n$, if the dividing line between $F_i$ and $G_i$ is between columns $c_i$ and 
$c_i+1$, then the topmost cell of $S$ in column $c_i$ is below the topmost cell of $S$ in column 
$c_i+1$. Similarly, if the dividing line between $G_i$ and $F_{i+1}$ is between rows $r_i$ and 
$r_i+1$, then the rightmost cell of $S$ in row $r_i$ is to the left of the rightmost cell of $S$ in 
row $r_i+1$.
\item[(d)] For $i<n$, $G_i$ is to the left of the vertical line separating $F_{i+1}$ from 
$G_{i+1}$, and $F_i$ is below the horizontal line separating $G_i$ from~$F_{i+1}$. 
\end{itemize}
We call this representation of $S$ the \emph{Ferrers decomposition} of~$S$. We remark that the 
Ferrers decomposition bears a certain similarity to the concept of staircase decomposition, 
introduced in the study of pattern-avoiding permutations~\cite{albert,GV}.

\begin{figure}[htb]
 \centerline{\includegraphics[scale=0.7]{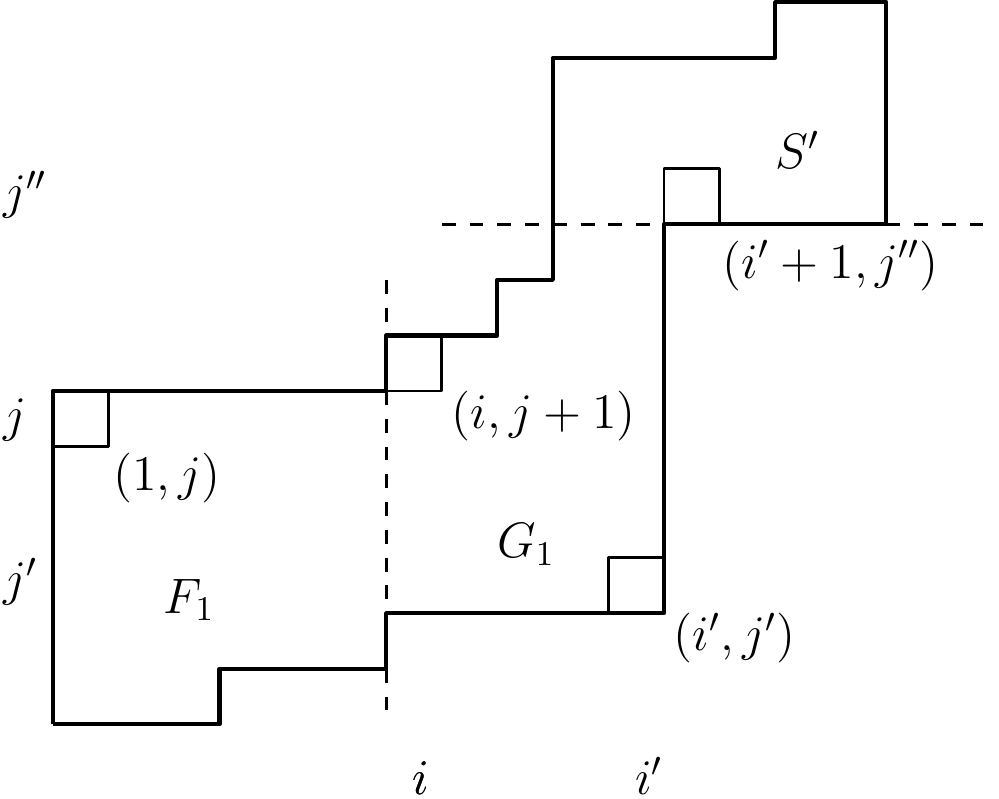}}
\caption{Decomposing a $\ds$-free skew shape into Ferrers diagrams}\label{fig-decomp}
\end{figure}

\begin{lemma} \label{lemma_key}
A connected skew shape is $\ds$-free if and only if it is Ferrers-decomposable.
\end{lemma}
\begin{proof}
In a Ferrers-decomposable skew shape $S$, every cell is in either a $NW$ Ferrers subshape or a $SE$ 
Ferrers subshape and thus satisfies the condition of Lemma~\ref{lemma_forb_rec}. Therefore $S$ is 
$\ds$-free. 

To prove the converse, let $S$ be a connected $\ds$-free skew shape; see Figure~\ref{fig-decomp}. 
We proceed by induction over the number of cells of $S$, noting that the statement is trivial if $S$ 
has a single cell. Let $(1,j)$ be the topmost cell in the first column of $S$. If $j$ is the topmost 
row of $S$, then $S$ is a NW Ferrers shape, and it admits the trivial Ferrers decomposition 
$S=F_1$. 

Suppose that $j$ is not the topmost row of $S$. Since $S$ is connected, it contains a cell in row 
$j+1$, and let $(i,j+1)$ be the leftmost such cell. We define $F_1$ to be the subshape of $S$ 
formed by the cells in the first $i-1$ columns. Clearly, $F_1$ is a nonempty NW Ferrers shape. 
Since $S$ is connected, the cell $(i,j)$ is in $S$. Applying Lemma~\ref{lemma_forb_rec} to this 
cell, we conclude that $S[\ge i,\le j]$ is a rectangle -- note that $S[\le i, \ge j]$ cannot be a 
rectangle, since $S$ contains the cells $(i-1,j)$ and $(i,j+1)$ but not $(i-1,j+1)$.

Let $(i',j')$ be the bottom-right cell of the rectangle $S[\ge i,\le j]$. If $i'$ is the rightmost 
column of $S$, then $S\setminus F_1$ is a SE Ferrers shape $G_1$, and we have a Ferrers 
decomposition $S=F_1|^v G_2$. Suppose that $i'$ is not the rightmost column, and let $(i'+1,j'')$ 
be the bottommost cell of $S$ in column $i'+1$. Note that $j''>j$, otherwise $S[\ge i,\le j]$ would 
not be a rectangle. We then let $G_1$ be the subshape of $S$ formed by the cells in columns 
$i,i+1,\dotsc,i'$ and rows $j',\dotsc,j''-1$. $G_1$ is then a SE Ferrers shape, and $S$ can be 
written as $S=F_1|^v G_1 |^h S'$, where $S'$ is a connected $\ds$-free skew shape. 

By induction, $S'$ admits a Ferrers decomposition $S'=F_2|^v G_2|^h\dotsb|^v G_n$. We claim that 
the expression $F_1 |^v G_1|^h F_2|^v G_2|^h\dotsb|^v G_n$ is a Ferrers decomposition of $S$. 
Clearly, it 
satisfies conditions (a), (b) and (c) in the definition of Ferrers decomposition. To verify 
condition (d), we only need to argue that the dividing line between $F_2$ and $G_2$ is to the right 
of~$G_1$. To see this, we note that by Lemma~\ref{lemma_forb_rec}, $S[\le i',\ge j'']$ is a 
rectangle, and in particular, the topmost cells in the columns intersecting this rectangle are all 
in the same row. Thus, by condition (c) of Ferrers decomposition, this rectangle cannot be 
separated by a vertical dividing line, and in particular the line separating $F_2$ from $G_2$ is to 
the right of~$G_1$.
\end{proof}

\subsection{Rubey's bijection} \label{sec3}

Rubey~\cite{Rubey11} has proved several general bijective results about fillings of moon polyominoes. 
Here we formulate a part of his results which will be useful in our argument. 

Given a finite sequence $s = (s_1, s_2, \ldots, s_n)$ and a permutation $\sigma$ of length $n$,
we let $\sigma s$ be the sequence $(s_{\sigma(1)}, s_{\sigma(2)}, \ldots, s_{\sigma(n)})$. In 
addition, given a moon polyomino $M$ of width $w$ and a permutation $\pi$ of length $w$, we 
denote by $\sigma M$ the shape created by permuting the columns of $M$ according to $\pi$. A 
\emph{maximal rectangle} in a moon polyomino $M$ is a rectangle $R\subseteq M$ which is not a proper 
subset of any other rectangle contained in~$M$. Note that a moon polyomino contains a maximal 
rectangle of width $w$ if and only if it has a row with exactly $w$ cells, and this maximal rectangle 
is then unique. The analogous property holds for columns as well. In fact, if $M$ and $M'$ are moon 
polyominoes such that $M$ is obtained from $M'$ by permuting its columns, then there is a 
size-preserving correspondence between the maximal rectangles of $M$ and the maximal rectangles 
of~$M'$; see Rubey~\cite{Rubey11}.

\begin{definition}
Let $M$ be a moon polyomino of height $h$ and width $w$, let $\bm{r}$ and $\bm{c}$ be sequences of 
integers of lengths $h$ and $w$ respectively, and let $\Lambda$ be a mapping which assigns to every 
maximal rectangle $R$ in $M$ a nonnegative integer $\Lambda(R)$.
Then $\F^{NE}(M, \Lambda, \bm{r}, \bm{c})$ is the set of all integer fillings of $M$ such that
\begin{itemize}
\item the sum of the entries in the $j$-th row of $M$ is equal to $r_j$,
\item the sum of the entries in the $i$-th column of $M$ is equal to $c_i$, and
\item for every maximal rectangle $R$ of $M$, the length of the longest $NE$-chain in the filling 
of $R$ is equal to~$\Lambda(R)$.
\end{itemize}
\end{definition}

The following theorem is a special case of a result of Rubey~\cite[Theorem 5.3]{Rubey11}.

\begin{theorem}[Rubey~{\cite{Rubey11}}]\label{thm_rubey}
Let $M$ be a moon polyomino of height $h$ and width $w$, and let $\sigma$ be a permutation of its 
columns such that $\sigma M$ is again a moon polyomino. Let $\Lambda$ be a function assigning 
nonnegative integers to maximal rectangles of $M$, let $\bm{r}$ and $\bm{c}$ be nonnegative integer 
vectors of length $h$ and $w$, respectively. Then the two sets $\F^{NE}(M, \Lambda, \bm{r}, \bm{c})$ 
and $\F^{NE}(\sigma M, \Lambda, \bm{r}, \sigma\bm{c})$ have the same size, and there is an explicit 
bijection between them.
\end{theorem}

\subsection{A bijection on Ferrers fillings}
Using Rubey's bijection as our main tool, we will now state and prove a technical result on 
fillings of Ferrers shapes satisfying certain restrictions.

Let $F$ be a NW Ferrers shape, and let $k$ and $\ell$ be integers. Suppose that the leftmost $k$ 
columns of $F$ all have the same length, and that the topmost $\ell$ rows of $F$ all have the same 
length. We will call the leftmost $k$ columns of $F$ \emph{special columns} and the topmost $\ell$ 
rows \emph{special rows} of~$F$. For any $i\le k$, let $C_i$ denote the rectangle formed by the 
$i$ leftmost special columns of $F$ (i.e., columns $1,\dotsc,i$), and let $C'_i$ be the rectangle 
formed by the $i$ rightmost special columns of $F$ (i.e., columns $k-i+1, k-i+2,\dotsc,k$). 
Similarly, for $j\le \ell$, $R_j$ is the rectangle formed by the $j$ topmost special rows of $F$, 
and $R'_j$ the rectangle formed by the $j$ bottommost special rows of~$F$. Additionally, for $i\le 
k$, let $c_i$ and $c'_i$ denote the $i$-th special column from the left and the $i$-th special 
column from the right, respectively (i.e., $c_i=i$ and $c'_i=k+1-i$), and for $j\le\ell$, let $r_j$ 
and $r'_j$ be the $j$-th special row from the top and the $j$-th special row from the bottom, 
respectively (i.e., $r_j=h+1-j$ and $r'_j=h-\ell+j$, where $h$ is the height of $F$).

\begin{lemma}\label{lem-ferrers}
With $F$, $k$ and $\ell$ as above, there is a bijection $\Gamma$ on the set of fillings of $F$, such 
that for any filling $\phi$ of $F$ and its image $\phi'=\Gamma(\phi)$, the following holds.
\begin{enumerate}
\item The longest SE-chain in $\phi$ has the same length as the longest NE-chain in~$\phi'$.
\item For every $i\le k$, the longest SE-chain of $\phi$ contained in $C_i$ has the same length as 
the longest NE-chain of $\phi'$ contained in~$C'_i$.
\item For every $j\le \ell$, the longest SE-chain of $\phi$ contained in $R_j$ has the same length 
as the longest NE-chain of $\phi'$ contained in~$R'_j$.
\item The entries of $\phi$ in column $c_i$ have the same sum as the entries of $\phi'$ in column 
$c'_i$, and the entries of $\phi$ in row $r_j$ have the same sum as the entries of $\phi'$ in 
row~$r'_j$. 
\item For any non-special row or column, the sum of its entries in $\phi$ is the same as the sum of 
its entries in~$\phi'$.
\end{enumerate}
\end{lemma}
\begin{proof}
Let $\phi$ be a filling of~$F$. Refer to Figure~\ref{fig-steps} for the illustration of the steps 
of the construction. As the first step, we add new cells to $F$, filled with zeros, as follows: for 
every $i\le k$, we add $i$ new cells to the bottom of column $c'_i$, and for every $j\le \ell$, we 
add $j$ new cells to the right of row~$r'_j$. We refer to these newly added cells as \emph{dummy 
cells}. We let $F_1$ denote the Ferrers shape obtained from $F$ by the addition of the dummy cells, 
and $\phi_1$ the corresponding filling.

\begin{figure}
 \centerline{\includegraphics[scale=0.7]{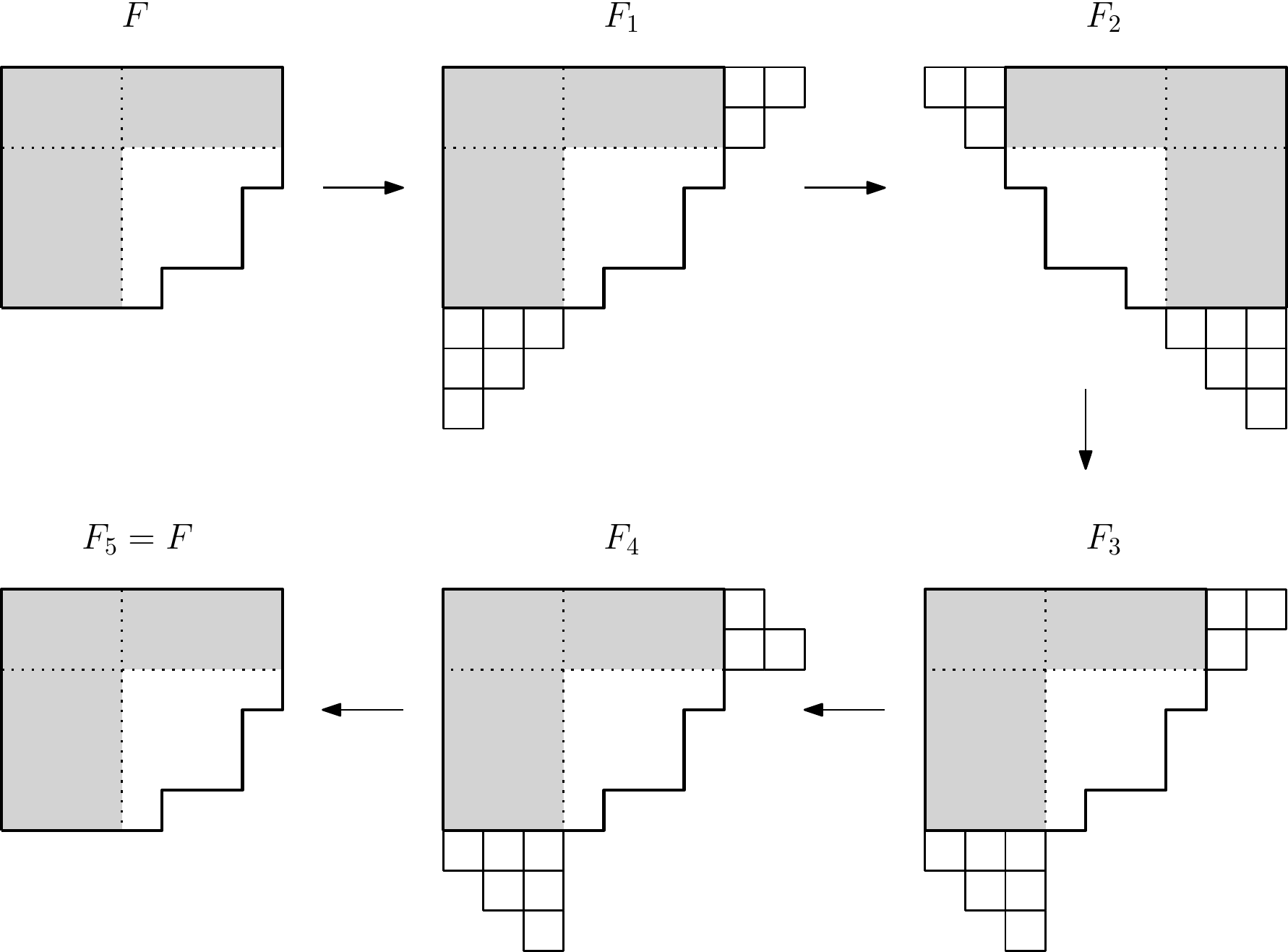}}
\caption{Construction of the bijection $\Gamma$ from Lemma~\ref{lem-ferrers}. The shaded areas 
correspond to the special rows and special columns.}\label{fig-steps}
\end{figure}

The dummy cells will remain filled with zeros throughout all the steps of the construction and will 
not contribute to any chain; their purpose is to ensure that for every $i\le k$ and $j\le\ell$, the 
shape $F_1$ has a maximal rectangle containing the cells of $C_i$ and no other cell of $F$, and 
similarly for~$R_j$. 

As the next step, we reverse the order of the columns of $F_1$ and $\phi_1$, while keeping the order 
of the rows preserved. Thus, we replace the filling $\phi_1$ of $F_1$ with its `mirror image' 
$\phi_2$ of the shape $F_2$, changing all SE-chains into NE-chains.

Let $\sigma$ be the permutation of the columns of $F_2$ which leaves the special columns of $F_2$ 
fixed, and places the non-special columns to the right of the special ones, in the same order as 
in~$F_1$. This permutation of columns yields a moon polyomino that only differs from $F_1$ by having 
the order of the special columns reversed. We call this polyomino~$F_3$. We now apply to $\phi_2$ the 
bijection from Theorem~\ref{thm_rubey}, with $\sigma$ as above, and let $\phi_3$ be the resulting 
filling of~$F_3$.

Next, we permute the rows of $F_3$, by reversing the order of the special rows. Let $F_4$ be the 
resulting moon polyomino. The corresponding filling $\phi_4$ is obtained by again invoking 
Theorem~\ref{thm_rubey}, or more precisely its symmetric version for row permutations.

Finally, we remove from $F_4$ and $\phi_4$ the dummy cells, to obtain a filling $\phi_5$ of the 
original Ferrers shape~$F$. We define $\Gamma$ by $\Gamma(\phi)=\phi_5$. It follows from 
Theorem~\ref{thm_rubey} that $\Gamma$ has all the properties stated in the lemma.
\end{proof}

\subsection{Proof of Theorem \ref{thm-sskew}}
We now have all the ingredients to prove the main result of this section.

\begin{proof}[Proof of Theorem \ref{thm-sskew}]
Clearly, it suffices to prove the theorem for connected skew shapes, since for a disconnected skew 
shape, we may treat each component separately. Let $S$ be a $\ds$-free connected skew shape. 
Fix its Ferrers decomposition 
\[
S = F_1 |^v G_1 |^h F_2 |^v \cdots |^h F_n |^v G_n.
\]
For two consecutive Ferrers shapes $F_i$ and $G_i$ from the decomposition, say that a row
is \emph{special} if $S$ contains a cell from both $F_i$ and $G_i$ in this row. In particular, $F_i$ 
and $G_i$ have the same number of special rows. Clearly, all the special rows have the same length 
in $F_i$ as well as in~$G_i$. Similarly, for consecutive Ferrers shapes $G_i$ and $F_{i+1}$, define 
a special column to be any column containing cells from both these shapes. 

Fix now a filling $\phi$ of $S$. We transform $\phi$ by applying the bijection $\Gamma$ separately 
to the restriction of $\phi$ in every Ferrers shape of the decomposition. More precisely, in 
$F_1,\dotsc,F_n$ we apply $\Gamma$ directly, while in $G_1,\dotsc,G_n$ we first rotate the filling 
by $180$ degrees, apply $\Gamma$, and then rotate back. Let $\phi'$ be the resulting filling of $S$, 
obtained by combining the transformed fillings of the individual Ferrers shapes in the decomposition.

We claim that $\phi'$ has all the properties stated in Theorem~\ref{thm-sskew}. Suppose first that 
$\phi$ contains $\delta_k$ for some $k$, and let us show that $\phi'$ contains~$\iota_k$. If the 
occurrence of $\delta_k$ in $\phi$ is confined to a single shape $F_i$ or $G_i$, then $\phi'$ 
contains $\iota_k$ in the same shape by the properties of~$\Gamma$: indeed, note that a filling 
of the shape $F_i$ (or $G_i$) contains $\delta_k$ if and only if its longest SE-chain has length at 
least $k$, and it contains $\iota_k$ if and only if its longest NE-chain has length at least~$k$.

Suppose now that $\phi$ has an occurrence of $\delta_k$ not confined to a single Ferrers shape of the 
decomposition. The occurrence must then be confined to the union of two consecutive shapes of the 
decomposition, since a pair of shapes further apart does not share any common row or column. 
Suppose, without loss of generality, that $\phi$ has an occurrence of $\delta_k$ inside $F_i\cup 
G_i$, with $k_1$ 1-cells from the occurrence being in $F_i$ and $k_2=k-k_1$ 1-cells from the 
occurrence being in~$G_i$. Moreover, let $s$ be the number of special rows in $F_i$ (and therefore 
also in $G_i$). There are then some values $s_1,s_2$ with $s_1+s_2=s$, such that $F_i$ contains a 
SE-chain of length $k_1$ in its topmost $s_1$ special rows while $G_i$ has a SE-chain of length 
$k_2$ in its $s_2$ bottommost special rows (which will become topmost once $G_i$ is rotated before 
the application of $\Gamma$).

Consequently, in $\phi'$, the shape $F_i$ has a NE-chain of length $k_1$ in its bottommost $s_1$ 
special rows, and $G_i$ has a NE-chain of length $k_2$ in its topmost $s_2$ special rows. The two 
chains together form a NE-chain of size $k$, i.e., an occurrence of $\iota_k$ in~$S$. The same 
reasoning shows that an occurrence of $\iota_k$ in $\phi'$ implies an occurrence of $\delta_k$ 
in~$\phi$.

Let us now consider the row sums of $\phi$ and $\phi'$. Let $S_i$ be the set of row indices of the 
special rows shared by $F_i$ and $G_i$. The sets $S_1,S_2,\dotsc,S_n$ are pairwise disjoint. If a 
row $r$ does not belong to any $S_i$, then it intersects only one shape from the decomposition and 
is not special. Such a row then has the same sum in $\phi$ as in $\phi'$, by the properties 
of~$\Gamma$. On the other hand, for the rows in a set $S_i$, the properties of $\Gamma$ ensure that 
the order of their row sums is reversed by the mapping $\phi\mapsto\phi'$, that is, the sum of the 
topmost row of $S_i$ in $\phi$ is the same as the sum of the bottommost row of $S_i$ in~$\phi'$, 
and so on. Thus, the row sums of $\phi'$ are obtained from the row sums of $\phi$ by a permutation 
that reverses the order of each $S_i$ and leaves the remaining elements fixed. An analogous property 
holds for column sums.
\end{proof}

\section{Binary fillings avoiding a chain of length 2}\label{sec-genskew}

In this section, we prove Theorem~\ref{thm-genskew}. Recall the skew shape $\ds$ and its filling 
$\fd$ from Figure~\ref{fig-D}.

Given a skew shape $S$ with a total of $N$ cells, we assign labels $c_1, c_2, \ldots, c_N$ to 
every cell starting from the lower left
corner, iterating over rows from the bottom to the top of $S$ and labeling cells in a row from left 
to right, as indicated in Figure~\ref{figure_cell_labels}.

We associate with every cell $c$ a three-part piecewise linear curve $l(c)$ consisting of the ray 
going from the top-right corner of $c$ to the left, the right border of $c$ and the ray going from 
the bottom-right corner of $c$ to the right. The curve $l(c_i)$ of the cell $c_i$ of a skew shape 
$S$ clearly divides $S$ into two parts, with cells $c_1,\dotsc,c_i$ below and $c_{i+1},\dotsc,c_N$ 
above the curve, as illustrated in Figure~\ref{figure_cell_labels}.

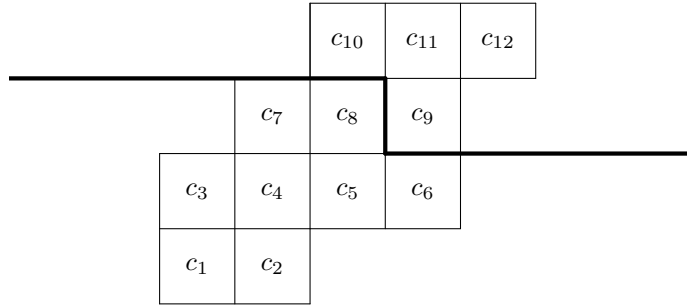
\begin{figure}[ht]
\centering
\begin{tikzpicture}[line cap=round,line join=round,>=triangle 45,x=1.0cm,y=1.0cm]
\clip(-6,-1.14) rectangle (3,3.15);
\draw (-4,-1)-- (-2,-1);
\draw (-2,-1)-- (-2,0);
\draw (-2,0)-- (0,0);
\draw (0,0)-- (0,2);
\draw (0,2)-- (1,2);
\draw (1,2)-- (1,3);
\draw (1,3)-- (-2,3);
\draw (-2,3)-- (-2,2);
\draw (-2,2)-- (-3,2);
\draw (-3,2)-- (-3,1);
\draw (-3,1)-- (-4,1);
\draw (-4,1)-- (-4,-1);
\draw (-2,3)-- (-2,0);
\draw (-1,3)-- (-1,0);
\draw (0,3)-- (0,2);
\draw (0,2)-- (-2,2);
\draw (0,1)-- (-3,1);
\draw (-2,0)-- (-4,0);
\draw (-3,1)-- (-3,-1);
\draw (-3.5, -0.5) node {$c_1$};
\draw (-2.5, -0.5) node {$c_2$};
\draw (-3.5, 0.5) node {$c_3$};
\draw (-2.5, 0.5) node {$c_4$};
\draw (-1.5, 0.5) node {$c_5$};
\draw (-0.5, 0.5) node {$c_6$};
\draw (-2.5, 1.5) node {$c_7$};
\draw (-1.5, 1.5) node {$c_8$};
\draw (-0.5, 1.5) node {$c_9$};
\draw (-1.5, 2.5) node {$c_{10}$};
\draw (-0.5, 2.5) node {$c_{11}$};
\draw (0.5, 2.5) node {$c_{12}$};
\draw [line width=1.6pt] (-1,2) -- (-10,2);
\draw [line width=1.6pt] (-1,1) -- (10,1);
\draw [line width=1.6pt] (-1,2) -- (-1,1);
\end{tikzpicture}
\caption{A skew shape with cell labels and the curve $l(c_8)$}
\label{figure_cell_labels}
\end{figure}

Let $\pi\in\{\delta_2,\iota_2,\fd\}$ be a filling. Let $S$ be a skew shape with cells numbered 
$c_1,\dotsc, c_N$ as above, and let $c_i$ be one of its cells. We say that an occurrence of $\pi$ in 
a filling $\psi$ of $S$ is \emph{$i$-low}, if the top-right cell of the occurrence is one of the 
cells $c_1,\dotsc,c_i$. Intuitively, this means that the entire occurrence is below the line 
$l(c_i)$. If an occurrence is not $i$-low, then it is $i$-high. Note that an occurrence may be 
$i$-high even when all its 1-cells are below~$l(c_i)$.

For a skew shape $S$ with cells numbered $c_1,\dotsc, c_N$ as above, let $\cG_i$ be the set of all 
the binary fillings of $S$ that avoid any $i$-high occurrence of $\delta_2$ as well as any $i$-low 
occurrence of $\iota_2$ or~$\fd$. See Figure~\ref{fig-forb} for an example.

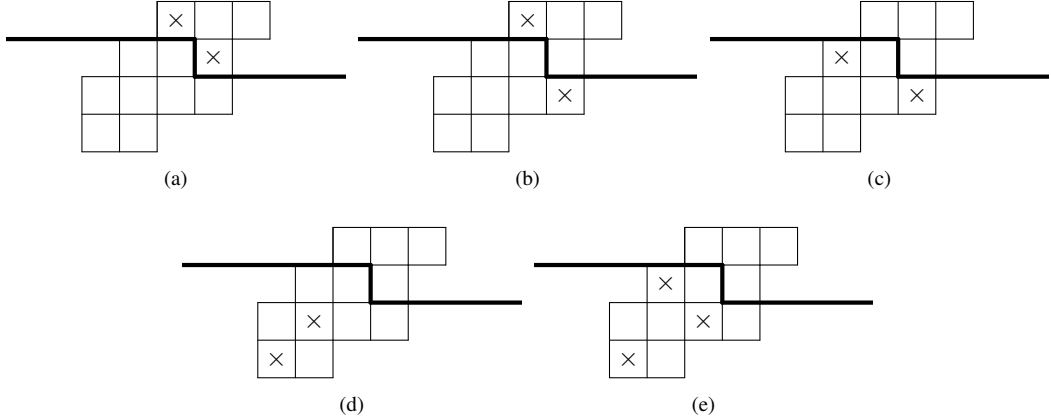
\begin{figure}[ht]
\centering
\subfloat[][] {
\begin{tikzpicture}[line cap=round,line join=round,>=triangle 45,x=0.5cm,y=0.5cm]
\clip(-6,-1.14) rectangle (3,3.15);
\draw (-4,-1)-- (-2,-1);
\draw (-2,-1)-- (-2,0);
\draw (-2,0)-- (0,0);
\draw (0,0)-- (0,2);
\draw (0,2)-- (1,2);
\draw (1,2)-- (1,3);
\draw (1,3)-- (-2,3);
\draw (-2,3)-- (-2,2);
\draw (-2,2)-- (-3,2);
\draw (-3,2)-- (-3,1);
\draw (-3,1)-- (-4,1);
\draw (-4,1)-- (-4,-1);
\draw (-2,3)-- (-2,0);
\draw (-1,3)-- (-1,0);
\draw (0,3)-- (0,2);
\draw (0,2)-- (-2,2);
\draw (0,1)-- (-3,1);
\draw (-2,0)-- (-4,0);
\draw (-3,1)-- (-3,-1);
\draw [line width=1.6pt] (-1,2) -- (-10,2);
\draw [line width=1.6pt] (-1,1) -- (10,1);
\draw [line width=1.6pt] (-1,2) -- (-1,1);
\draw (-1.5, 2.5) node {$\times$};
\draw (-0.5, 1.5) node {$\times$};
\end{tikzpicture}
}
\subfloat[][] {
\begin{tikzpicture}[line cap=round,line join=round,>=triangle 45,x=0.5cm,y=0.5cm]
\clip(-6,-1.14) rectangle (3,3.15);
\draw (-4,-1)-- (-2,-1);
\draw (-2,-1)-- (-2,0);
\draw (-2,0)-- (0,0);
\draw (0,0)-- (0,2);
\draw (0,2)-- (1,2);
\draw (1,2)-- (1,3);
\draw (1,3)-- (-2,3);
\draw (-2,3)-- (-2,2);
\draw (-2,2)-- (-3,2);
\draw (-3,2)-- (-3,1);
\draw (-3,1)-- (-4,1);
\draw (-4,1)-- (-4,-1);
\draw (-2,3)-- (-2,0);
\draw (-1,3)-- (-1,0);
\draw (0,3)-- (0,2);
\draw (0,2)-- (-2,2);
\draw (0,1)-- (-3,1);
\draw (-2,0)-- (-4,0);
\draw (-3,1)-- (-3,-1);
\draw [line width=1.6pt] (-1,2) -- (-10,2);
\draw [line width=1.6pt] (-1,1) -- (10,1);
\draw [line width=1.6pt] (-1,2) -- (-1,1);
\draw (-1.5, 2.5) node {$\times$};
\draw (-0.5, 0.5) node {$\times$};
\end{tikzpicture}
}
\subfloat[][] {
\begin{tikzpicture}[line cap=round,line join=round,>=triangle 45,x=0.5cm,y=0.5cm]
\clip(-6,-1.14) rectangle (3,3.15);
\draw (-4,-1)-- (-2,-1);
\draw (-2,-1)-- (-2,0);
\draw (-2,0)-- (0,0);
\draw (0,0)-- (0,2);
\draw (0,2)-- (1,2);
\draw (1,2)-- (1,3);
\draw (1,3)-- (-2,3);
\draw (-2,3)-- (-2,2);
\draw (-2,2)-- (-3,2);
\draw (-3,2)-- (-3,1);
\draw (-3,1)-- (-4,1);
\draw (-4,1)-- (-4,-1);
\draw (-2,3)-- (-2,0);
\draw (-1,3)-- (-1,0);
\draw (0,3)-- (0,2);
\draw (0,2)-- (-2,2);
\draw (0,1)-- (-3,1);
\draw (-2,0)-- (-4,0);
\draw (-3,1)-- (-3,-1);
\draw [line width=1.6pt] (-1,2) -- (-10,2);
\draw [line width=1.6pt] (-1,1) -- (10,1);
\draw [line width=1.6pt] (-1,2) -- (-1,1);
\draw (-2.5, 1.5) node {$\times$};
\draw (-0.5, 0.5) node {$\times$};
\end{tikzpicture}
}

\subfloat[][] {
\begin{tikzpicture}[line cap=round,line join=round,>=triangle 45,x=0.5cm,y=0.5cm]
\clip(-6,-1.14) rectangle (3,3.15);
\draw (-4,-1)-- (-2,-1);
\draw (-2,-1)-- (-2,0);
\draw (-2,0)-- (0,0);
\draw (0,0)-- (0,2);
\draw (0,2)-- (1,2);
\draw (1,2)-- (1,3);
\draw (1,3)-- (-2,3);
\draw (-2,3)-- (-2,2);
\draw (-2,2)-- (-3,2);
\draw (-3,2)-- (-3,1);
\draw (-3,1)-- (-4,1);
\draw (-4,1)-- (-4,-1);
\draw (-2,3)-- (-2,0);
\draw (-1,3)-- (-1,0);
\draw (0,3)-- (0,2);
\draw (0,2)-- (-2,2);
\draw (0,1)-- (-3,1);
\draw (-2,0)-- (-4,0);
\draw (-3,1)-- (-3,-1);
\draw [line width=1.6pt] (-1,2) -- (-10,2);
\draw [line width=1.6pt] (-1,1) -- (10,1);
\draw [line width=1.6pt] (-1,2) -- (-1,1);
\draw (-3.5, -0.5) node {$\times$};
\draw (-2.5, 0.5) node {$\times$};
\end{tikzpicture}
}
\subfloat[][] {
\begin{tikzpicture}[line cap=round,line join=round,>=triangle 45,x=0.5cm,y=0.5cm]
\clip(-6,-1.14) rectangle (3,3.15);
\draw (-4,-1)-- (-2,-1);
\draw (-2,-1)-- (-2,0);
\draw (-2,0)-- (0,0);
\draw (0,0)-- (0,2);
\draw (0,2)-- (1,2);
\draw (1,2)-- (1,3);
\draw (1,3)-- (-2,3);
\draw (-2,3)-- (-2,2);
\draw (-2,2)-- (-3,2);
\draw (-3,2)-- (-3,1);
\draw (-3,1)-- (-4,1);
\draw (-4,1)-- (-4,-1);
\draw (-2,3)-- (-2,0);
\draw (-1,3)-- (-1,0);
\draw (0,3)-- (0,2);
\draw (0,2)-- (-2,2);
\draw (0,1)-- (-3,1);
\draw (-2,0)-- (-4,0);
\draw (-3,1)-- (-3,-1);
\draw [line width=1.6pt] (-1,2) -- (-10,2);
\draw [line width=1.6pt] (-1,1) -- (10,1);
\draw [line width=1.6pt] (-1,2) -- (-1,1);
\draw (-3.5, -0.5) node {$\times$};
\draw (-2.5, 1.5) node {$\times$};
\draw (-1.5, 0.5) node {$\times$};
\end{tikzpicture}
}
\caption{Forbidden patterns of $\mathcal{G}_8$}\label{fig-forb}
\end{figure}

Note that $\cG_1$ is precisely the set of all $\delta_2$-avoiding binary fillings of $S$,
while $\cG_N$ is the set of $\{\iota_2,\fd\}$-avoiding binary fillings of $S$. Thus, 
Theorem~\ref{thm-genskew} will follow directly from the next lemma.

\begin{lemma}\label{lemma_gi}
Let $S$ be a skew shape with $N$ cells and let $1 \leq i < N$. Then there is a bijection between 
$\cG_i$ and~$\cG_{i+1}$. Moreover, this bijection preserves the number of 1-cells in each row.
\end{lemma}

Most of the remainder of this section is devoted to the proof of Lemma~\ref{lemma_gi}. 
First, consider the case when the cells $c_i$ and $c_{i+1}$ are not in the same row, i.e., 
$c_{i+1}$ 
is the first cell of a new row and $c_i$ is the last cell of the previous row. In this case we will 
show that the sets $\cG_i$ and $\cG_{i+1}$ are identical. 

To see this, note that since the three relevant patterns $\delta_2$, $\iota_2$ and $\fd$ all have 
two cells in the topmost row, they can never have an occurrence whose top-right cell is mapped to 
$c_{i+1}$. Thus, in a filling $\psi$ of $S$, an occurrence of any of these patterns is $i$-low
if and only if it is $(i+1)$-low, and hence $\cG_i=\cG_{i+1}$.

\begin{figure}[ht]
\centering
\includegraphics[width=0.5\textwidth]{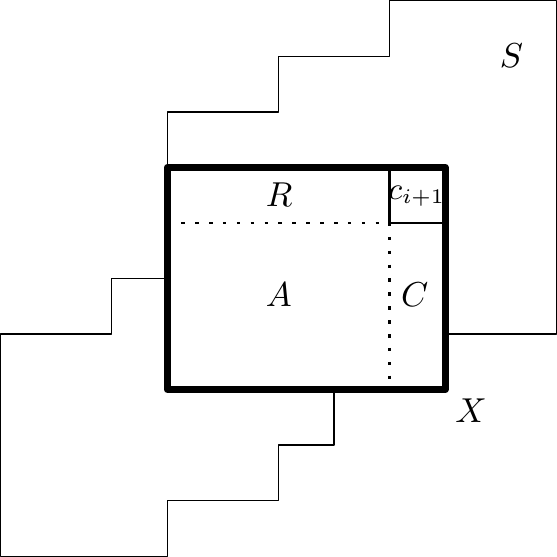}
% \begin{tikzpicture}[line cap=round,line join=round,>=triangle 45,x=0.75cm,y=0.75cm]
% \clip(0.67,0.63) rectangle (14.32,12.34);
% \draw (1,1)-- (5,1);
% \draw (5,1)-- (5,2);
% \draw (5,2)-- (8,2);
% \draw (8,2)-- (8,4);
% \draw (8,4)-- (11,4);
% \draw (11,4)-- (11,6);
% \draw (11,6)-- (13,6);
% \draw (13,6)-- (13,9);
% \draw (13,9)-- (14,9);
% \draw (14,9)-- (14,12);
% \draw (1,1)-- (1,4);
% \draw (1,4)-- (2,4);
% \draw (2,4)-- (2,8);
% \draw (2,8)-- (4,8);
% \draw (4,8)-- (4,11);
% \draw (4,11)-- (8,11);
% \draw (8,11)-- (8,12);
% \draw (8,12)-- (14,12);
% \draw (2,8)-- (9,8);
% \draw [dash pattern=on 2pt off 2pt] (9,8)-- (9,4);
% \draw [dash pattern=on 2pt off 2pt] (9,4)-- (2,4);
% \draw [dash pattern=on 2pt off 2pt] (8,4)-- (8,8);
% \draw [dash pattern=on 2pt off 2pt] (2,7)-- (8,7);
% \draw (5, 7.5) node {$R$};
% \draw (5, 5.5) node {$A$};
% \draw (8.5, 5.5) node {$C$};
% \draw (8.5, 7.5) node {$c_{i+1}$};
% \draw [line width=1.6pt] (2,8) -- (9,8);
% \draw [line width=1.6pt] (9,8) -- (9,7);
% \draw [line width=1.6pt] (8,8) -- (8,7);
% \draw [line width=1.6pt] (8,7) -- (13,7);
% \end{tikzpicture}
\caption{The maximal rectangle $X$ with $c_{i+1}$ in the top-right corner}
\label{figure_max_rec}
\end{figure}

Suppose now that $c_i$ and $c_{i+1}$ are adjacent cells in a single row.  Let $R$ consist of all the 
cells in the same row as $c_{i+1}$ and 
strictly to the left of it, let $C$ consist of all the cells in the column of $c_{i+1}$ and 
strictly below it, and let $A$ be the rectangle consisting of all cells that are in a column 
intersecting $R$ and a row intersecting~$C$; see Figure~\ref{figure_max_rec}. Note that $A$, $R$, 
$C$ and $c_{i+1}$ form a rectangle $X$ which is maximal among all the rectangles contained in $S$ 
with $c_{i+1}$ as top-right corner. Therefore, in any occurrence of $\iota_2$ with the upper 1-cell 
in $c_{i+1}$, the lower 1-cell is inside the rectangle~$A$.

Let us make several observations about the structure of the fillings in $\cG_i$ and $\cG_{i+1}$ 
which will be useful later in the proof.

\begin{observation}\label{obs-rc}
In a filling $\psi\in\cG_i$, at most one of $R$ and $C$ can be nonzero, otherwise $\psi$ would 
contain an $i$-high copy of $\delta_2$.
\end{observation}

\begin{observation}\label{obs-ac}
In a filling $\rho\in\cG_{i+1}$, at most one of $c_{i+1}$ and $A$ can be nonzero, otherwise $\rho$ 
would contain an $(i+1)$-low copy of~$\iota_2$.
\end{observation}

\begin{observation}\label{obs-acra} In a filling $\psi\in\cG_i\cup\cG_{i+1}$, there is no copy of 
$\iota_2$ inside either of the two rectangles $A\cup C$ and $A\cup R$, since such a copy would 
necessarily be $i$-low (and hence also $(i+1)$-low). In particular, if $C_a$ and $C_b$ are two 
nonzero columns of $A\cup C$ (or $A\cup R$) with $C_a$ to the left of $C_b$, then all the 1-cells 
of $c_a$ are weakly above any 1-cell of~$C_b$.
\end{observation}

We say that in a filling $\psi\in\cG_i\cup\cG_{i+1}$, the row $R$ \emph{overlaps} the rectangle $A$ 
in column $C'$, if the column $C'$ contains a 1-cell in row $R$ as well as a 1-cell inside~$A$. Note 
that if this happens, then row $R$ contains no 1-cell to the right of column $C'$, and $A$ contains 
no 1-cell to the left of column $C'$, otherwise we get a contradiction with 
Observation~\ref{obs-acra}.

We will partition each of the sets $\mathcal{G}_i$ and $\mathcal{G}_{i+1}$ into five pairwise
disjoint subsets and construct bijections between corresponding subsets. The set $\cG_i$
is partitioned as follows:
\begin{itemize}
\item $\mathcal{G}_i^1$ contains the fillings from $\mathcal{G}_i$ in which $c_{i+1}=0$ or~$A=0$.
\item $\mathcal{G}_i^2$ contains the fillings from $\mathcal{G}_i$ with $c_{i+1}=1$, $A\neq 0$ and 
 $C\neq0$. Note that this implies $R=0$, by Observation~\ref{obs-rc}.
\item $\cG_i^3$ contains the fillings from $\mathcal{G}_i$ with $c_{i+1}=1$, $R$ overlapping with 
$A$, and $A$ having only one nonzero column. Note that this implies $C=0$, by 
Observation~\ref{obs-rc}.
\item $\cG_i^4$ contains the fillings from $\mathcal{G}_i$ with $c_{i+1}=1$, $R$ overlapping with 
$A$, and $A$ having at least two nonzero columns. Note that this again implies $C=0$, by 
Observation~\ref{obs-rc}.
\item $\cG_i^5$ contains the fillings from $\mathcal{G}_i$ with $c_{i+1}=1$, $A\neq 0$,  
$R$ not overlapping $A$ (possibly $R=0$), and $C=0$.
\end{itemize}

Note that each filling from $\cG_i$ belongs to exactly one of the five subsets defined above. We now 
define an analogous partition of $\cG_{i+1}$.

\begin{itemize}
\item $\mathcal{G}_{i+1}^1$ contains the fillings from $\mathcal{G}_{i+1}$ in which $C=0$ or $R=0$.
\item $\mathcal{G}_{i+1}^2$ contains the fillings from $\mathcal{G}_{i+1}$ in which $C\neq 0$, 
$R$ contains exactly one 1-cell, and $R$ overlaps with $A$. Note that this implies $A\neq 0$ and 
hence $c_{i+1}=0$ by Observation~\ref{obs-ac}.
\item $\cG_{i+1}^3$ contains the fillings from $\mathcal{G}_{i+1}$ in which $c_{i+1}=1$, and both $R$ 
and $C$ are nonzero. Note that this implies $A=0$, by Observation~\ref{obs-ac}.
\item $\cG_{i+1}^4$ contains the fillings from $\mathcal{G}_{i+1}$ in which $C\neq0$, 
$R$ contains more than one 1-cell, and $R$ overlaps with~$A$. This again implies that $A\neq 0$ 
and hence $c_{i+1}=0$ by Observation~\ref{obs-ac}.
\item $\cG_{i+1}^5$ contains the fillings from $\mathcal{G}_{i+1}$ in which $R$ and $C$ are both 
nonzero, $c_{i+1}=0$, and $R$ does not overlap~$A$ (possibly with $A=0$).
\end{itemize}

Again, it follows from the definitions that each filling from $\cG_{i+1}$ belongs to exactly one 
of the five sets above. Our plan is to show that for every $j\in\{1,\dotsc,5\}$, the two sets 
$\cG_i^j$ and $\cG_{i+1}^j$ have the same size. We first deal with the sets $\cG_i^1$ and 
$\cG_{i+1}^1$, which turn out to be identical.

\begin{claim}\label{cla-1}
The set $\cG_i^1$ is equal to the set $\cG_{i+1}^1$.
\end{claim}
\begin{proof}
To show that $\cG_i^1\subseteq\cG_{i+1}^1$, choose a filling $\psi\in\cG_i^1$. Clearly, since $\psi$ 
has no $i$-high occurrence of $\delta_2$, it has no $(i+1)$-high occurrence either. Also, $C$ or $R$ 
must be zero in $\psi$, otherwise $\psi$ would contain an $i$-high occurrence of $\delta_2$.
The only possibility for $\psi$ to have an $(i+1)$-low occurrence of $\iota_2$ or $\fd$
is that $c_{i+1}$ is actually the top-right corner of the occurrence. Since $c_{i+1}$ or $A$ is 
zero by the definition of $\cG_i^1$, there is no such occurrence of~$\iota_2$. Since $C$ or $R$ is 
zero as we saw above, there is no such occurrence of $\fd$ either. This shows that $\psi$ 
belongs to~$\cG_{i+1}^1$. 

Conversely, to show that $\cG_{i+1}^1\subseteq\cG_i^1$, choose $\rho\in \cG_{i+1}^1$. We see 
that $\rho$ has no $i$-low occurrence of $\iota_2$ or $\fd$, since it has no such $(i+1)$-low 
occurrence. It also has no $i$-high occurrence of $\delta_2$, because the top-right cell of such an 
occurrence would have to coincide with $c_{i+1}$, which is excluded by the definition of 
$\cG_{i+1}^1$. Hence $\rho$ is in $\cG_i$. And since $\rho$ has no $(i+1)$-low occurrence of 
$\iota_2$, we may conclude that $\rho$ is in~$\cG_i^1$. This proves the claim.
\end{proof}

Let us define $\cgi^+=\cgi\setminus\cgi^1$ and $\cgip^+=\cgip\setminus\cgip^1$. 
To show that $|\cG_i^j|=|\cG_{i+1}^j|$ for $j>1$, we will construct a bijection $f\colon 
\cgi^+\to\cgip^+$ and its inverse $g\colon\cgip^+\to\cgi$, such that for every $j\in\{2,3,4,5\}$ and 
every $\psi\in\cgi^j$ we have $f(\psi)\in \cgip^j$, and similarly, for any $\rho\in\cgip^j$, we 
have $g(\rho)\in\cgi^j$.

We first describe the two mappings $f$ and $g$, and then show that they have the required 
properties. In all the cases, the mappings $f$ and $g$ will only modify the values of the cells 
inside the rectangle $X$, while the rest of the filling remains unchanged.

For a filling $\psi\in\cgi^2$, let $C'\subseteq A$ be the leftmost nonzero column of~$A$. The 
filling $f(\psi)$ is then obtained from $\psi$ by setting the cell of $R$ directly above $C'$ to 1, 
and the cell $c_{i+1}$ to 0, keeping the remaining cells of $\psi$ unchanged. 

Consider now a filling $\rho\in\cgip^2$. Recall from the definition of $\cgip^2$ and from 
Observation~\ref{obs-acra} that $R$ contains a single 1-cell of $\rho$, and this 1-cell is directly 
above the leftmost nonzero column of~$A$. We define $g(\rho)$ to be the filling obtained from 
$\rho$ by setting the single 1-cell of $R$ to 0 and the cell $c_{i+1}$ to~1.

For a filling $\psi\in\cgi^3$, let $C'\subseteq A$ be the unique nonzero column of~$A$. Then 
$f(\psi)$ is obtained from $\psi$ by replacing the filling of~$C$ with the filling of~$C'$, and then 
filling $C'$ with zeros. 

For a filling $\rho\in\cgip^3$, we let $C'\subseteq A$ be the column of $A$ directly below the 
rightmost 1-cell of $R$, and we define $g(\rho)$ as the filling obtained from $\rho$ by replacing the 
filling of $C'$ with the values from~$C$, and filling $C$ with zeros.

Consider now $\psi\in\cgi^4$. Let $C_1, C_2,\dotsc,C_k\subseteq A$ be the sequence of all the 
nonzero columns of $A$, in left-to-right order. Recall that by the definition of $\cgi^4$ and by 
Observation~\ref{obs-acra}, the rightmost 1-cell of $R$ is directly above the column~$C_1$. We 
create the filling $f(\psi)$ from $\psi$ as follows: we replace the filling of $C$ by the filling of 
$C_k$, then we replace, for every $j=1,\dotsc,k-1$, the filling of $C_{j+1}$ by the filling of 
$C_j$, and we fill $C_1$ with zeros. Finally, we set cell of $R$ directly above $C_2$ to 1, and the 
cell $c_{i+1}$ to~0.

Now consider a filling $\rho\in\cgip^4$. We let $C_2, C_3,\dotsc,C_k$  be the nonzero columns of $A$ 
ordered left to right. Note that by the definition of $\cgip^4$ and by Observation~\ref{obs-acra}, 
the rightmost 1-cell in $R$ is directly above~$C_2$, and $R$ has at least one other 1-cell to the 
left of it. Let $C_1$ be the rightmost column of $A$ to the left of $C_2$ that has a 1-cell of 
$R$ directly above it. We create $g(\rho)$ from $\rho$ as follows: for $j=1,\dotsc,k-1$ we 
fill $C_j$ with the values in $C_{j+1}$, then fill $C_k$ with the values in $C$, fill $C$ with 
zeros, and finally set the cell of $R$ directly above $C_2$ to 0 and the cell $c_{i+1}$ to~1.

For a filling $\psi\in\cgi^5$, let $C_1, C_2,\dotsc,C_k\subseteq A$ be again the sequence of all the 
nonzero columns of $A$ in left-to-right order. If $R$ has any 
1-cells, they must all be strictly to the left of the cell above~$C_1$. We create $f(\psi)$ from 
$\psi$ by replacing $C$ with $C_k$, then replacing $C_{j+1}$ with $C_j$ for every $j=1,\dots,k-1$, 
and filling $C_1$ with zeros. Finally, we set the cell of $R$ directly above $C_1$ to 1 and 
$c_{i+1}$ to~0. 

For a filling $\rho\in\cgip^5$, let $C_2,\dotsc,C_k$ be the nonzero columns of $A$ 
ordered left to right, and let $C_1$ be the column of $A$ directly below the rightmost 1-cell of $R$. 
Note that by the definition of $\cgip^5$, $C_1$ is strictly to the left of~$C_2$. We create $g(\rho)$ 
by replacing, for $j=1,\dotsc,k-1$, $C_j$ with $C_{j+1}$, replacing $C_k$ with $C$, filling $C$ with
zeros, setting the cell of $R$ above $C_1$ to 0 and $c_{i+1}$ to~1.

Our next aim is to show that the two mappings $f$ and $g$ defined above cannot create any of the 
forbidden patterns. We first collect several useful observations about the properties of $f$ 
and~$g$, which can be verified by inspecting the definition of the two mappings.

\begin{observation}\label{obs-fg} 
Let $\psi$ be a filling from $\cgi^+$, and let $\rho$ be a filling from $\cgip^+$. Recall that 
$X$ is the rectangle formed by the union of $C$, $R$, $A$ and $c_{i+1}$.
\begin{enumerate}
\item \label{p-row} Each row of $X$ has the same number of 1-cells in $\psi$ as in $f(\psi)$, and 
the same number 
of 1-cells in $\rho$ as in $g(\rho)$.
\item \label{p-col} A column of $X$ is nonzero in $\psi$ if and only if it is nonzero in $f(\psi)$. 
A column of $X$ is nonzero in $\rho$ if and only if it is nonzero in~$g(\rho)$.
\item \label{p-r} Suppose that $f(\psi)$ has a 1-cell $c'$ inside $R$. Then all the 
1-cells of $f(\psi)$ inside $A$ are weakly to the right of~$c'$. The same property holds 
for~$g(\rho)$ as well.
\item \label{p-frow} Suppose that $(x,y)$ is a 1-cell of $f(\psi)$ inside $A\cup C$. Then $\psi$ 
contains a 1-cell $(x',y)\in A\cup C$ with $x'\le x$, i.e., a 1-cell in $A\cup C$ in the same row 
and weakly to the left of~$(x,y)$.
\item \label{p-fcol} Suppose that $(x,y)$ is a 1-cell of $f(\psi)$ inside $A\cup R$. Then $\psi$ has 
a 1-cell $(x,y')\in A\cup R$ with $y'\le y$, i.e., a 1-cell in $A\cup R$ in the same column and 
weakly below~$(x,y)$.  (Here we use Observation~\ref{obs-acra}.)
\item \label{p-grc} Suppose that $(x,y)$ is a 1-cell of $g(\rho)$ inside $A$. Then $\rho$ contains a 
1-cell $(x',y)$ in $A\cup C$ for some $x'\ge x$, as well as a 1-cell $(x,y')$ in $A\cup R$ for some 
$y'\ge y$. (This again follows from Observation~\ref{obs-acra}.)
\item \label{p-ac} The subfilling of $\psi$ induced by its nonempty columns of $A\cup C$  is the 
same as the subfilling of $f(\psi)$ induced by its nonempty columns of $A\cup C$. The same is true 
for the two fillings $\rho$ and $g(\rho)$.
\item \label{p-grr} The mapping $g$ does not change any 0-cell in $R\cup C$ into a 1-cell. In 
other words, if a cell $c'\in R\cup C$ is a 1-cell in $g(\rho)$ then it is also a 1-cell in~$\rho$.
\end{enumerate}
\end{observation}

\begin{claim}\label{cla-f} 
For any $\psi\in\cgi^+$, the filling $f(\psi)$ belongs to $\cgip$.
\end{claim}
\begin{proof}
If $f(\psi)$ has an $(i+1)$-high copy of $\delta_2$, then at least one 1-cell of this copy must be 
in~$X$, because in the rest of $S$, no 1-cells were modified. The other 1-cell must be strictly 
above or strictly to the right of $X$, otherwise the occurrence would be $(i+1)$-low.

However, by parts~\ref{p-row} and~\ref{p-col} of Observation~\ref{obs-fg}, for every 1-cell of 
$f(\psi)$ in $X$, there is a 1-cell of $\psi$ in the same column of $X$ and also a 1-cell of $\psi$ 
in the same row of~$X$. It follows that to any $(i+1)$-high occurrence of $\delta_2$ in $f(\psi)$, 
we may also find an $(i+1)$-high (and therefore $i$-high) occurrence in $\psi$ as well, 
contradicting $\psi\in \cG_i$.

Suppose now that $f(\psi)$ has an $(i+1)$-low occurrence of $\iota_2$, formed by two 1-cells $u$ 
and $v$, where $u$ is to the left and below~$v$. Since the occurrence is $(i+1)$-low, the cell $v$ 
is either in $R\cup\{c_{i+1}\}$ or its row is below the row containing~$R$. It follows that the row 
of $u$ is below the row of~$R$. This leaves the following cases to consider.
\begin{itemize} 
\item $u\in A$ and $v\in R$: this contradicts Observation~\ref{obs-fg} part~\ref{p-r}.
\item $u\in A$ and $v\in A\cup C$: by Observation~\ref{obs-fg} part~\ref{p-ac},  $\psi$ also 
contains an occurrence of $\iota_2$ in $A\cup C$, which is impossible.
\item $u\in A\cup C$ and $v\not\in X$: consequently, $v$ is to the right of $X$ and the row 
containing $v$ intersects~$A$. Then $v$ is also a 1-cell of $\psi$, and by Observation~\ref{obs-fg} 
part \ref{p-row}, $\psi$ has a 1-cell $u'$ in $A$ in the same row as $u$. Then $u'$ and $v$ form an 
$i$-low copy of $\iota_2$ in $\psi$, a contradiction.
\item $u\in A$ and $v=c_{i+1}$: this is impossible, since $c_{i+1}$ is a 1-cell in $f(\psi)$ only if 
$\psi$ is in $\cgi^3$, and in that case~$A=0$.
\item $u\not\in X$:  then $v$ is in $X$, otherwise $u$ and $v$ would form 
a forbidden $\iota_2$ in~$\psi$. Consequently, $u$ is in a row that is below the bottommost row of 
$X$, or in a column that is to the left of the leftmost column of~$X$. If the row containing $u$ is 
below the bottommost row of $X$, then the column containing $C$ and $c_{i+1}$ does not intersect the 
row of $u$ inside the shape $S$ (recall that $X$ is the largest rectangle inside $S$ with top-right 
corner $c_{i+1}$). Hence $v$ is in $A\cup R$ and by Observation~\ref{obs-fg} part~\ref{p-fcol}, 
$\psi$ has a 1-cell $v'\in A\cup R$ in the same column as $v$ and weakly below it. Then $u$ and 
$v'$ form 
a forbidden~$\iota_2$ in $\psi$, a contradiction. Similarly, if the column of $u$ is to the left of 
$X$, then $v$ is in $A\cup C$ and by Observation~\ref{obs-fg} part~\ref{p-frow}, $\psi$ has a 
1-cell $v'\in A\cup C$ in the same row as $v$ and weakly to the left of it, forming again the 
forbidden 
$\iota_2$ with~$u$.
\end{itemize}
The above cases cover all possibilities, and show that $f(\psi)$ avoids an $(i+1)$-low~$\iota_2$.

\begin{figure}[ht]
\centering
\includegraphics[width=0.8\textwidth]{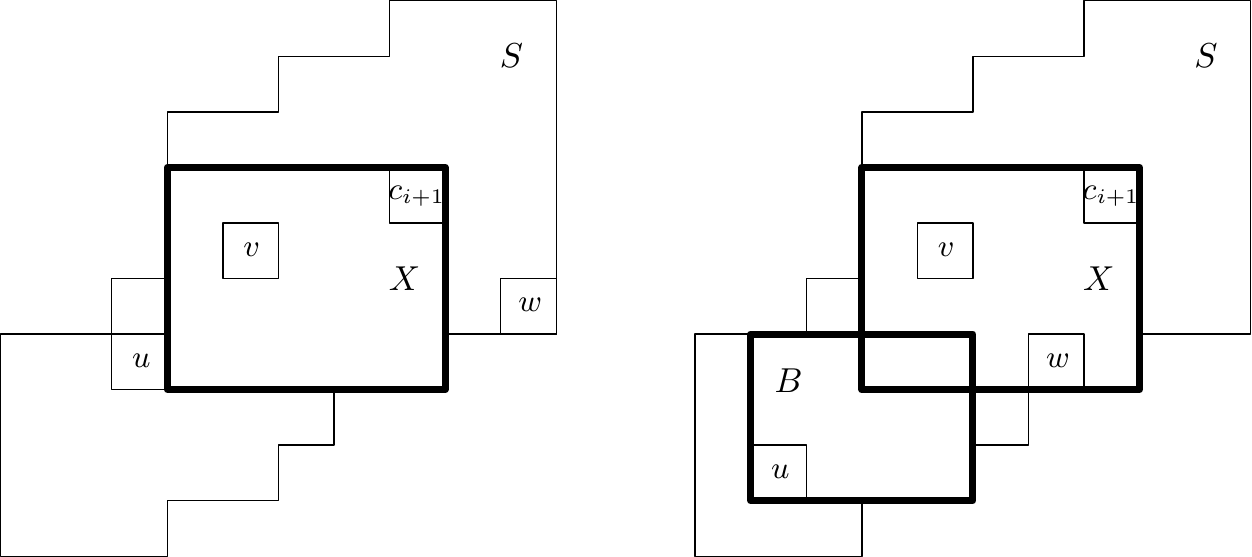}
\caption{An occurrence of $\fd$ in the cells $u$, $v$ and $w$. Left: the situation when the row of 
$u$ intersects~$X$. Right: the situation when the row of $u$ is below~$X$.}
\label{figure_sforb_occ}
\end{figure}

Suppose now that $f(\psi)$ has an $(i+1)$-low occurrence of $\fd$. Let $u$, $v$ and $w$ be the three 
1-cells of this occurrence, ordered left to right. We note that in 
order to get an occurrence of $\fd$, the column containing $u$ must not intersect the row containing 
$v$ inside~$S$, and the row containing $u$ must not intersect the column of $w$ inside~$S$. Since 
all the three 1-cells are below the line $l(c_{i+1})$, this implies that $u$ is strictly to the left 
of the leftmost column of~$A$, otherwise the column of $u$ would intersect the row of $v$ 
inside~$S$. Moreover, at least one of $v$ and $w$ must be inside $X$, otherwise 
the three 1-cells would form a forbidden copy of $\fd$ in~$\psi$. 

Suppose first that the row containing $u$ intersects~$X$; see Figure~\ref{figure_sforb_occ} left. 
Then $w$ must be strictly to the right of 
$X$, otherwise the column of $w$ would intersect the row of $u$ inside~$S$, which cannot happen, as 
we pointed out above. This implies that $v$ is in $X$, and more precisely $v$ must be in $A\cup C$, 
otherwise the copy of $\fd$ would not be $(i+1)$-low. But then, 
by Observation~\ref{obs-fg} part~\ref{p-frow},  $\psi$ has a 1-cell $v'$ in the same row as $v$ and 
weakly to its left, and $\{u,v',w\}$ is a forbidden occurrence of $\fd$ 
in~$\psi$, a contradiction. This shows that the row of $u$ cannot intersect~$X$.

It remains to consider the case where the row of $u$ is strictly below the bottommost row of~$X$; 
see Figure~\ref{figure_sforb_occ} right.
Define $B$ to be the largest rectangle inside $S$ whose bottom-left corner is~$u$. Note that $\psi$ 
has no 1-cell in $X\cap B$, since such a 1-cell would form a forbidden occurrence of $\iota_2$ 
with~$u$. Note also that the column containing $v$ intersects the row containing $w$ inside $B$, 
otherwise $u$, $v$ and $w$ would not form a copy of~$\fd$. 

By Observation~\ref{obs-fg} parts~\ref{p-frow} and~\ref{p-fcol}, 
the filling $\psi$ contains a 1-cell $v'$ in the same column as $v$ and weakly below it, and it 
contains a 1-cell $w'$ in the same row as $w$ and weakly to the left of it. Since $\psi$ has no 
1-cell inside $B$, neither $v'$ nor $w'$ are in $B$.  Moreover, since the bottom-left corner of $B$ 
is to the left and below $X$, we know that $v'$ is above $B$ and $w'$ is to the right of $B$.
Hence $u$, $v'$ and $w'$ form an $(i+1)$-low occurrence of $\fd$ in~$\psi$. 
In fact, this occurrence must be $i$-low as well, otherwise we would have $v'\in R$ and 
$w'\in C$, in which case $v'$ and $w'$ would form a forbidden $i$-high occurrence of $\delta_2$ 
in~$\psi$. We conclude that $\psi$ has a forbidden occurrence of~$\fd$. This contradiction 
shows that $f(\psi)$ has no $(i+1)$-low occurrence of~$\fd$, completing the proof of 
Claim~\ref{cla-f}.
\end{proof}

\begin{claim}\label{cla-g}
For any $\rho\in\cgip^+$, the filling $g(\rho)$ belongs to $\cgi$.
\end{claim}
\begin{proof}
Assume first that $g(\rho)$ has an $i$-high copy of~$\delta_2$, formed by two cells $u$ and $v$, 
with $u$ above and to the left of $v$. If the top-right corner of the copy of $\delta_2$ is 
$c_{i+1}$, then $u\in R$ and $v\in C$. However, this is impossible, since the column $C$ of 
$g(\rho)$ is only nonzero when $\rho$ is in $\cgip^2$, and in that case we have $R=0$. If follows 
that the copy of $\delta_2$ is $(i+1)$-high, and in particular, at most one of the two cells $u$ and 
$v$ is in~$X$.  If neither $u$ nor $v$ is in $X$, then $u$ and $v$ form an $(i+1)$-high copy of 
$\delta_2$ in $\rho$, which is impossible. If $u$ is in $X$ and $v$ is not, then $v$ is to the right 
of $X$, and by Observation~\ref{obs-fg} part~\ref{p-row}, $\rho$ has a 1-cell $u'\in X$ in the same 
row of $X$ as $u$. Then $\{u',v\}$ forms a forbidden copy of $\delta_2$ in~$\rho$. If, on the other 
hand, $v$ is in $X$, then $u$ is above $X$, and by Observation~\ref{obs-fg} part~\ref{p-col}, 
$\rho$ has a 1-cell $v'\in X$ in the same column as $v$. Now $\{u,v'\}$ forms a forbidden copy of 
$\delta_2$ in~$\rho$. In all cases we get a contradiction, concluding that $g(\rho)$ has no $i$-high 
copy of~$\delta_2$.

Suppose now that $g(\rho)$ contains an $i$-low copy if~$\iota_2$, and let $u$ and $v$ be its two 
1-cells, with $u$ left and below~$v$. At least one of the two 1-cells must be in $X$, otherwise they 
would form an $i$-low (and therefore $(i+1)$-low) copy of $\iota_2$ in~$\rho$. We distinguish 
several cases:
\begin{itemize}
\item $u\in A$ and $v\in R$: this is excluded by Observation~\ref{obs-fg} part~\ref{p-r}.
\item $u\in A$ and $v\in A\cup C$: then Observation~\ref{obs-fg} part~\ref{p-ac} implies that $\rho$ 
also has a forbidden~$\iota_2$.
\item $u\in X$ and $v\not\in X$: then $v$ must be to the right of $X$, and since the copy of 
$\iota_2$ is $i$-low, the row containing $v$ intersects~$X$. By Observation~\ref{obs-fg} 
part~\ref{p-row}, $\rho$ has a 1-cell $u'$ of $X$ in the same row as $u$, and this 1-cell forms a 
forbidden $\iota_2$ with~$v$.
\item $u\not \in X$ and the row containing $u$ intersects $X$: then $u$ is to the left of $X$,  $v$ 
is in $X$, and by Observation~\ref{obs-fg} part \ref{p-row}, $\rho$ has a 1-cell $v'$ in the same 
row 
as $v$, forming a forbidden copy of $\iota_2$ with~$u$.
\item $u\not \in X$ and the column containing $u$ intersects $X$: this is analogous to the previous 
case, except we use part~\ref{p-col} of Observation~\ref{obs-fg}.
\item $u\not\in X$ and neither the row nor the column of $u$ intersects~$X$: then $u$ is strictly 
bottom-left from $X$, and $v$ is in~$X$.  By part~\ref{p-grc} of Observation~\ref{obs-fg}, $\rho$ 
contains a 1-cell $v'\in X$ in the same column as $v$ and weakly above it, as well as a 
1-cell $v''\in X$ in the same row as $v$ and weakly to the right of it. Let $B$ be the largest 
rectangle inside $S$ whose bottom-left corner is~$u$. If $v'$ or $v''$ is inside $B$, it forms a 
forbidden copy of $\iota_2$ in~$\rho$. Otherwise $v'$ is strictly above $B$ and $v''$ strictly to its 
right, and $\{u,v',v''\}$ is a forbidden copy of $\fd$ in~$\rho$.
\end{itemize}
We conclude that $g(\rho)$ has no $i$-low copy of~$\iota_2$.

Finally, suppose $g(\rho)$ contains an $i$-low copy of $\fd$, and let $u$, $v$ and $w$ be 
its 1-cells ordered left to right. The column containing $u$ does not intersect the row of $v$, and 
the row of $u$ does not intersect the column of $w$ by the definition of~$\fd$. Since the three 
cells are below the line $l(c_i)$, it follows that the column of $u$ is to the left of the leftmost 
column of~$X$, otherwise it would intersect the row of~$v$. In particular, $u$ is also a 1-cell 
in~$\rho$.

We will now show that $\rho$ contains a pair of 1-cells $v',w'$ which form an $(i+1)$-low copy of 
$\fd$ with $u$ or an $(i+1)$-high copy of $\delta_2$, in both cases contradicting $\rho\in\cG_{i+1}$. 

We define $v'$ as follows: if $v$ is a 1-cell of $\rho$, we set $v'=v$. If, on the other hand, $v$ is 
a 0-cell in $\rho$ this means that $v$ is in $A$: indeed, since $g$ modified the value of $v$, we 
know that $v$ is in $X$, moreover, by part~\ref{p-grr} of Observation~\ref{obs-fg}, we have 
$v\not\in R\cup C$; in addition, $v$ is below $\ell(c_i)$ so $v\neq c_{i+1}$. By 
part~\ref{p-grc} of Observation~\ref{obs-fg}, $\rho$ contains a 1-cell of $A\cup R$ in the same 
column as $v$ and weakly above $v$; we then let $v'$ be one such 1-cell.

Similarly, if $w$ is a 1-cell in $\rho$, we let $w'=w$. If $w$ is a 0-cell in $\rho$, then 
necessarily $w$ is in~$A$. By part~\ref{p-grc} of Observation~\ref{obs-fg}, there is a 1-cell of 
$\rho$ in the same row as $w$ and to its right, and we let $w'$ be one such 1-cell.

Since $v'$ is in the same column as $v$ and weakly above it, while $w'$ is in the same row as $w$ 
and weakly to the right of it, it follows that $u$, $v'$ and $w'$ together form a copy 
of~$\fd$. If this copy is $(i+1)$-high, then $v'$ and $w'$ by themselves form an $(i+1)$-high copy 
of $\delta_2$, otherwise we have an $(i+1)$-low copy of~$\fd$. In both cases, we get a contradiction 
with~$\rho\in\cG_{i+1}$.

It follows that $g(\rho)$ is in $\cG_i$, proving Claim~\ref{cla-g}.
\end{proof}

With the help of Claims~\ref{cla-f} and \ref{cla-g}, it is routine to check, for any fixed 
$j\in\{2,3,4,5\}$, that for every $\psi\in\cgi^j$, the filling $f(\psi)$ is in $\cgip^j$ and 
$g(f(\psi))=\psi$, while for every $\rho\in\cgip^j$, $g(\rho)$ is in $\cgi^j$ and 
$f(g(\rho))=\rho$. This means that $f$ and $q$ are mutually inverse bijections between $\cgi^+$ and 
$\cgip^+$. By construction, both $f$ and $g$ preserve the number of 1-cells in each row.
Together with Claim~\ref{cla-1}, this completes the proof of Lemma~\ref{lemma_gi}.

\begin{proof}[Proof of Theorem \ref{thm-genskew}]
Let $N$ be the number of cells of $S$. From Lemma~\ref{lemma_gi}, we immediately get that there is a 
bijection between the fillings of $\mathcal{G}_1$ and $\mathcal{G}_N$. As discussed above, 
$\mathcal{G}_1$ is exactly the set of fillings avoiding $\delta_2$, while $\mathcal{G}_N$
is exactly the set of fillings avoiding $\fd$ and~$\iota_2$. Moreover, all the individual bijections 
we constructed have the property that they preserve row sums. This completes the proof of 
Theorem~\ref{thm-genskew}.
\end{proof}

\section{Conclusion}
Our main motivation for the study of fillings, and fillings of skew shapes in particular, stems from 
their close connection with the problem of enumeration of pattern-avoiding permutations. In 
particular, the skew shape conjecture offers a plausible way of obtaining an infinite family of 
Wilf-ordering relations. Specifically, the conjecture implies that for any two permutations $\alpha$ 
and $\beta$ and any integer $k$, the permutation pattern $\alpha\ominus\delta_k\ominus\beta$ is 
avoided by fewer permutations than the pattern $\alpha\ominus\iota_k\ominus\beta$; see 
\cite[Section 3.2]{phd} and references therein. Although the results of this paper do not seem to 
yield any direct consequences for permutation enumeration, we believe that a better understanding of 
pattern-avoiding skew-shape fillings may pave the way for future progress in this area.

\acknowledgements We thank the two anonymous referees, whose careful reading and insightful 
comments have helped is improve this paper.

\bibliographystyle{plain}
\bibliography{skew}

\end{document}